\documentclass[perprint]{elsarticle}

\usepackage{amssymb}
\usepackage{amsthm}
\usepackage{amsmath}
\usepackage{hyperref}
\usepackage{xypic}
\usepackage{color}

\newtheorem{SATZ}{Theorem}[section]

\newtheorem{LEMMA}[SATZ]{Lemma}
\newtheorem{DEF}[SATZ]{Definition}

\newtheorem{BEISPIEL}[SATZ]{Example}

\newtheorem{KOR}[SATZ]{Corollary}

\newtheorem{BEM}[SATZ]{Remark}

\newtheoremstyle{bare}        
  {}            
  {}            
  {\normalfont}                 
  {}                            
  {\bfseries}                   
  {}                            
  {.0em}                           
  {\thmnumber{#2}. \thmnote{ \normalfont\textsc{(#3)}}} 

\theoremstyle{bare}
\newtheorem{PAR}[SATZ]{}

\newcommand{\M}[1]{ \left(\begin{matrix} #1 \end{matrix} \right) }

\newcommand{\R}{ \mathbb{R} }
\newcommand{\C}{ \mathbb{C} }
\newcommand{\Q}{ \mathbb{Q} }
\newcommand{\Qp}{ {\mathbb{Q}_p} }

\newcommand{\Z}{ \mathbb{Z} }
\newcommand{\Zh}{ {\widehat{\mathbb{Z}}} }
\newcommand{\Zp}{ {\mathbb{Z}_p} }

\newcommand{\Zpp}{ {\mathbb{Z}_{(p)}} }

\newcommand{\Gm}{ {\mathbb{G}_m} }

\newcommand{\N}{ \mathbb{N} }
\newcommand{\HH}{ \mathbb{H} }

\newcommand{\A}{ \mathbb{A} }
\newcommand{\Af}{ {\mathbb{A}^{(\infty)}} }
\newcommand{\Afp}{ {\mathbb{A}^{(\infty, p)}} }

\newcommand{\PP}{ \mathbb{P} }
\newcommand{\SSS}{ \mathbb{S} }
\newcommand{\OOO}{\text{\footnotesize$\mathcal{O}$}}

\newcommand{\where}{\enspace | \enspace }

\newcommand{\comment}[1]{}

\newcommand{\tensor}{\otimes}

\DeclareMathOperator{\supp}{supp}
\DeclareMathOperator{\dd}{d}

\DeclareMathOperator{\Lie}{Lie}

\DeclareMathOperator{\im}{im}
\DeclareMathOperator{\ord}{ord}

\DeclareMathOperator{\Aut}{Aut}
\DeclareMathOperator{\Iso}{Iso}
\DeclareMathOperator{\Model}{M}

\DeclareMathOperator{\Cycle}{Z}

\DeclareMathOperator{\re}{Re}
\DeclareMathOperator{\Hom}{Hom}

\DeclareMathOperator{\Sym}{Sym}
\DeclareMathOperator{\End}{End}
\DeclareMathOperator{\Stab}{Stab}

\DeclareMathOperator{\Isome}{I}

\DeclareMathOperator{\ggT}{ggT}

\DeclareMathOperator{\SO}{SO}

\DeclareMathOperator{\hght}{ht}
\DeclareMathOperator{\vol}{vol}

\DeclareMathOperator{\pr}{pr}

\DeclareMathOperator{\tr}{tr}

\DeclareMathOperator{\Sp}{Sp}
\DeclareMathOperator{\SL}{SL}
\DeclareMathOperator{\GL}{GL}
\DeclareMathOperator{\PGL}{PGL}
\DeclareMathOperator{\Mp}{Mp}
\DeclareMathOperator{\adeg}{\widehat{deg}}

\DeclareMathOperator{\chern}{c}
\DeclareMathOperator{\achern}{\widehat{c}}

\DeclareMathOperator{\Div}{div}

\newcommand{\nP}{P}

\newcommand{\nX}{\mathbb{D}}
\newcommand{\nh}{h}

\newcommand{\nZ}{\Cycle}
\newcommand{\nSh}{\Model}
\newcommand{\nShD}{\Model^\vee}

\newcommand{\nweil}{\omega}

\newcommand{\WeilGamma}{\Upsilon}

\newcommand{\nRPCD}{\Delta}

\newcommand{\nL}{L}
\newcommand{\nM}{M}
\newcommand{\nN}{N}

\newcommand{\nO}{\mathbf{O}}

\newcommand{\nIsome}{\Isome}
\newcommand{\nInd}{\Isome}

\journal{Journal of Number Theory} 

\begin{document}

\begin{frontmatter}



\title{On recursive properties of certain $p$-adic Whittaker functions}


\author{Fritz H\"ormann}

\address{The Department of Mathematics and Statistics, McGill University, Montr\'eal, Canada\\
\emph{E-mail: hoermann@math.mcgill.ca}}

\begin{abstract}
We investigate recursive properties of certain $p$-adic Whittaker functions 
(of which representation densities of quadratic forms are special values).
The proven relations can be used to compute them explicitly in arbitrary dimensions, provided that enough information
about the orbits under the orthogonal group acting on the representations is available. These relations have implications for the 
first and second special derivatives of the 
Euler product over all $p$ of these Whittaker functions. These Euler products appear
as the main part of the Fourier coefficients of Eisenstein series associated with the Weil representation.
In case of signature $(m-2,2)$, we interpret these implications in terms of the theory of Borcherds' products on orthogonal Shimura varieties.
This gives some evidence for Kudla's conjectures in higher dimensions.
\end{abstract}


\end{frontmatter}



\section{Introduction}
\label{INTRO}

Let $\nL$ and $\nM$ be $\Z$-lattices of dimension $m$ and $n$ equipped with non-degenerate quadratic forms. 
The purpose of this article is the investigation of the associated representation densities, defined for each prime $p$ 
(in the simplest case) as the volume of
\[ \nIsome(\nM, \nL)(\Zp) = \{ \alpha: \nM_\Zp \rightarrow \nL_\Zp \where \alpha \text{ is an isometry} \}, \] 
w.r.t. some canonical volume form.
They are determined by the number of elements in 
$\nIsome(\nM, \nL)(\Z/p^k\Z)$
for sufficiently large $k$ and are of considerable interest because 
\begin{enumerate}
\item in the positive definite case, certain averages of the representation numbers $\# \nIsome(\nM, \nL^{(i)})(\Z)$ over all classes $\nL^{(i)}$
in the genus of $\nL$ are (up to a factor at $\infty$) equal to the Euler product over all representation densities by Siegel's formula. 
\item in the indefinite case, the same kind of product gives the relative volume of certain {\em special cycles} on the locally symmetric orbifold associated with the
lattice $\nL$. 
\item the product may be understood as a (special value of a) Fourier coefficient of an Eisenstein series associated with the Weil representation \cite{Weil1}.
This is related to 1., resp. 2., by the Siegel-Weil formula \cite{Weil2}. 
\end{enumerate}
The construction of Eisenstein series defines a natural `interpolation' of the individual factors in the Euler product as
a function in $s \in \C$. These functions are $p$-adic Whittaker functions in the non-Archimedean case and confluent hypergeometric functions in the Archimedean case, respectively.
All special values of the $p$-adic Whittaker functions at integral $s$ are volumes of sets like $\nIsome(\nM, \nL \oplus H^s)(\Zp)$, where $H$ denotes an hyperbolic plane.
Furthermore they are polynomials in $p^{-s}$ and hence determined by these values.
 
If $\nL$ has signature $(m-2,2)$, the locally symmetric spaces associated with $\nL$, mentioned in 2. above,  are in fact {\em Shimura varieties}.
Kudla, motivated by the celebrated work of Gross and Zagier \cite{GZ, GKZ}, conjectured a general relation of a {\em special derivative} w.r.t. $s$ of
the Fourier coefficients of the same Eisenstein series to {\em heights} of the special cycles (see e.g. \cite{Kudla5, Kudla6}). 
Defining and working appropriately with these heights requires integral models of compactifications of the Shimura varieties in question. 
In the thesis of the author \cite{Thesis}, a theory of those was developed and some of these conjectures could be partially verified in arbitrary dimensions. 
This required in particular a study of the recursive properties of the occurring $p$-adic Whittaker functions. 
This article is dedicated to a detailed discussion of those. It is organized as follows:

In section \ref{SECTIONREPDENS}, we define the representation densities $\mu_p(\nL_\Zp, \nM_\Zp, \kappa)$ 
as the volume of  $\nIsome(\nM, \nL)(\Qp) \cap \kappa$, where $\kappa$ is a coset in $(\nL_\Zp^*/\nL_\Zp) \otimes \nM_\Zp^*$, with respect to a certain
canonical measure. 
They differ from the classical densities by a discriminant factor, but have much nicer recursive
properties. All of these follow formally from the compatibility of the canonical measures with composition of maps (Theorem \ref{LEMMAVOLUMEFIBRATION}).
We recover a classical recursive property due to Kitaoka (Corollaries \ref{KITAOKA}, \ref{KITAOKA2}), 
as well as an easy orbit equation (Corollary \ref{elementaryorbiteq}) of the shape
\[ \mu_p(\nL_\Zp, \nM_\Zp, \kappa) =  \sum_{\text{orbits } \SO'(\nL_\Zp) \alpha \text{ in }\nIsome(\nM, \nL)(\Qp)\cap \kappa} \frac{ \text{volume of $\SO'(\nL_\Zp)$}}{\text{volume of $\SO'(\alpha^\perp_\Zp)$}}. \] 
where $\SO'$ denotes the discriminant kernel.
In section \ref{EIS}, we recall the Weil representation, the definition of the associated Eisenstein series and the relation of $p$-adic Whittaker functions to representation densities.

For the remaining part, we assume $p \not= 2$.

In section \ref{CONTINUATIONLAMBDAMU} ff., the volume of the discriminant kernel of the orthogonal group is `interpolated' as a function $\lambda(\nL_\Zp; s)$ by means of adding hyperbolic planes, too, which turns out to be a quite simple polynomial in $p^{-s}$ (Theorem \ref{EXPLIZIT}). 
We show that Kitaoka's formula and the above orbit equation are true also for the $p$-adic Whittaker functions $\mu_p(\nL_\Zp, \nM_\Zp, \kappa; s)$ and the $\lambda_p(\nL_\Zp; s)$. 
For this, we show that $\SO'$-orbits remain (ultimately) stable while adding hyperbolic planes.  
This reproves in an elementary way that the interpolated representation densities are polynomials in $p^{-s}$, too, 
for sufficiently large $s \in \Z_{\ge 0}$, and allows in principle to calculate these polynomials for arbitrary dimensions, 
provided one has enough information about the orbits. This is illustrated for $n=1$ in section
\ref{SIMPLEILLUSTRATION}, where we recover a special case of Yang's explicit formula (Theorem \ref{YANG}).

Section \ref{SECTIONN1} is dedicated to the case $\dim(\nM)=1$. In this case the $p$-adic Whittaker function, as a polynomial in $p^{-s}$, may be computed by means
of counting {\em all} $\nIsome(\nM, \nL)(\Z/p^k\Z)$ up to some specified $k$. This yields a relation to zeta functions of the lattice, too --- see Lemma \ref{REPDENS}. Furthermore, there exists a nice explicit
formula due to Yang in this case (Theorem \ref{YANG}). A similar formula is actually proven for $p=2$ in Yang's paper \cite{Yang1}, too. 

The development of the orbit equation, however, was motivated by the following. Assume $\nL_\Z$ is a global lattice of signature $(m-2,2)$.
Consider the locally symmetric space associated with $\nL_\Z$, the orbifold
\[ [ \SO(\nL_\Q) \backslash \nX_\nO \times (\SO(\nL_\Af) / \SO'(\nL_\Zh)) ], \]
where $\nX_\nO$ is the associated (Hermitian) symmetric space and $\SO'(\nL_\Zh)$ is the discriminant kernel, a compact open subgroup of
$\SO(\nL_\Af)$. It is a Shimura variety in this case, having a canonical model $\nSh({}^{\SO'(\nL_\Zh)} \nO)$ over $\Q$ ($m \ge 3$). On it, we have the {\em special 
cycle} $\nZ(\nL_\Z,\nM_\Z,\kappa)$, defined analytically as (see \ref{SPECIALCYCLE} for details)
\begin{gather*}
 \sum_{ \SO'(\nL_\Zh)\alpha \subset \nIsome(\nM,\nL)(\Af) \cap \kappa } \left[ \SO(\alpha^\perp_\Q) \backslash \nX_{\nO(\alpha^\perp)} \times (\SO(\alpha^\perp_\Af) / \SO'(\alpha^\perp_\Zh)) \right], 
\end{gather*}
i.e. as a sum of sub-Shimura varieties associated with certain lattices $\alpha^\perp_\Z$ (\ref{DEFI}).
The volume of the Shimura variety $\nSh({}^{\SO'(\nL_\Zh)} \nO)(\C)$ (w.r.t. a specific automorphic volume form) is given roughly by the 
product over all $\nu$ of $\lambda^{-1}_\nu(\nL_\Zp; 0)$ (with an appropriate factors for $\nu=\infty$).
Written in the product over all $\nu$ in the form \ref{GLOBALORBITEQUATION}:
\[ \lambda^{-1}(\nL_\Z; s) \mu(\nL_\Z, \nM_\Z, \kappa; s) = \sum_{ \SO'(\nL_\Zh)\alpha \subset \nIsome(\nM, \nL)(\Af) \cap \kappa} \lambda^{-1}(\alpha^\perp_\Z; s), \]
the value at $s=0$ expresses just the decomposition of the special cycle into sub-Shimura varieties (additivity of volume).
It is therefore tempting to believe, that its derivative should express the equality of the {\em height} of 
the special cycles as the sum of heights of its constituents. This however can {\em not} be true in general because
already for $n=1$ the $\mu$ in this equation differs from the
(holomorphic part) of the Fourier coefficient of the Eisenstein series by a factor of $|2 d(\nM_\Z)|^{-\frac{s}{2}}$ which comes from the Archimedian place (cf. Theorem \ref{SATZREPDENSARCHIMEDIANINTERPOLATION}). 
After incorporating a $|2 d(\nM_\Z)|_p^{\frac{s}{2}}$ into $\mu_p$, and slightly modifying $\lambda_p$, it remains true roughly at those $p$, where there is only one orbit (\ref{GLOBALORBITEQUATION}).
In the simple case of signature $(1,2)$, Witt rank 1 (modular curve) the correction in more ramified cases
can be explained by the existence of another, more deep local equation of this shape, related to the theta correspondence. This will be investigated briefly in section \ref{MODULARCURVE}.
Nevertheless, the orbit equation is already in this form technically very useful in the remaining cases, if one takes its derivative up to rational multiples of $\log(p)$, roughly for all $p$ where there is more than 1 orbit.

The application to Kudla's program of this and an interpretation of the value and derivative of the orbit equation at $s=0$ (case $n=1$) 
in terms of Borcherds' theory \cite{Borcherds1} is 
illustrated in section \ref{KUDLASPROGRAM}. It was the main achievement in the thesis of the author \cite{Thesis} and will be published in detail in a forthcoming paper \cite{Paper2}.
As motivation, we recommend the reader to read this section first, consulting only 
necessary definitions and statements from the foregoing text.

\section{Notation and basic definitions} 
\label{NOTATION}

\begin{PAR}
Let $R$ be a p.i.d., $\nL$ be an $R$-lattice, i.e. a free $R$-module of finite rank. 
We use the following notation:
\[
\begin{array}{rlcll}
\Sym^2(\nL^*)   &=& \{\text{quadratic forms on $\nL$} \}    &=& ((\nL\otimes \nL)^s)^* \\
\Sym^2(\nL)^*   &=& \{\text{symm. bilinear forms on $\nL$} \}   &=& (\nL^* \otimes \nL^*)^s \\
\Sym^2(\nL)     &=& \{\text{quadratic forms on $\nL^*$} \}  &=& ((\nL^* \otimes \nL^*)^s)^* \\
\Sym^2(\nL^*)^* &=& \{\text{symm. bilinear forms on $\nL^*$} \} &=& (\nL \otimes \nL)^s,
\end{array}
\]
here $(\cdots)^s$ denotes symmetric elements, i.e. invariants under the automorphism switching factors.
Usually a non-degenerate quadratic form in $\Sym^2(\nL^*)$ will be fixed and denoted by $Q_\nL$. Its associated bilinear form
$v,w \mapsto Q_\nL(v+w)-Q_\nL(v)-Q_\nL(w)$ is denoted by $\langle\cdot,\cdot\rangle_Q$, and its associated morphism $\nL \rightarrow \nL^*$ by 
$\gamma_Q$.

We denote the discriminant of $Q_\nL$, i.e. the determinant of
$\gamma_Q$ w.r.t. some basis of $\nL$ by $d(\nL)$. It is determined up to $(R^*)^2$. We denote by $<\varepsilon_1, \cdots, \varepsilon_m>$ the
lattice $R^m$ with quadratic form $x \mapsto \sum_{i=1}^m \varepsilon_i x_i^2$ ($R$ is always understood from the context). 
It has discriminant $2^m \prod_{i=1}^{m} \varepsilon_i$.  
\end{PAR}

\begin{PAR}\label{NATURALCHARACTERS}
In this article, we work with the following natural (up to a choice of $i \in \C$) characters on $R=\cdots$:
\begin{itemize}
\item[$\R$:] $\chi_\infty(x) := e^{2\pi i x}$,
\item[$\Qp$:] $\chi_p(x) := e^{-2 \pi i [x] }$, where $[x] = \sum_{i<0} x_i p^{-i}$ is the principal part, \\
(it has level/conductor 1),
\item[$\A^S$:] $\chi = \prod_{\nu\not\in S} \chi_\nu$.
\end{itemize}
The corresponding self-dual additive Haar measures are the
Lebesque measure on $\R$,  the standard measures on $\Qp$,
giving $\Zp$ the volume 1, and their product, respectively.
\end{PAR}

\begin{PAR}\label{WEILMEASURES}
Let $R$ be one of the rings of \ref{NATURALCHARACTERS}. 
Let $X$ be an algebraic variety over $R$ and $\widetilde{\mu}$ an algebraic volume form on $X$. 
As is explained in \cite{Weil3} (cf. also \cite[\S 3.5]{PlRa}), this defines a well-defined measure $\mu$ on $X(R)$, 
which depends on the choice of $\chi$ (resp. the additive Haar measure). For the special
case of a lattice $\nL$ of dimension $r$, and $\widetilde{\mu} \in \Lambda^r \nL^*$, 
there is $\widetilde{\mu}^* \in \Lambda^r \nL$ satisfying $\widetilde{\mu}^*\widetilde{\mu} = 1$.
In this case, the measures $\mu$ and $\mu^*$ are dual to each other with 
respect to the bicharacter $v, v^* \mapsto \chi(v^*v)$, i.e. for
\begin{eqnarray*}
 F_\Psi(w^*) &=& \int_\nL \Psi(w) \chi(w^* w) \mu(w) \qquad \Psi \in S(\nL)  \\
 F_\Psi(w) &=& \int_{\nL^*} \Psi(w^*) \chi(w^* w) \mu^*(w^*) \qquad \Psi \in S(\nL^*)
\end{eqnarray*}
(where $S(\cdots)$ denotes space of Schwartz-Bruhat functions), we have $F_{F_\Psi}(w^*) = \Psi(-w^*)$. 
\end{PAR}

\begin{PAR}\label{DEFCANVOL}
Let $\nL$ be an $R$-lattice with non-degenerate quadratic form $Q_\nL$. 
Then there is a canonical (translation invariant) measure $\mu_\nL$ with $\mu_\nL^* = \mu_\nL$ under
the identification $\gamma_Q: \nL \stackrel{\sim}{\longrightarrow} \nL^*$.
Let $e_1, \dots, e_m$ be a basis of $\nL$, $e_1^*, \dots, e_m^*$ the dual basis and
$\widetilde{\mu} = e_1^* \wedge \cdots \wedge e_m^*$. Let $A$ be the matrix of $\langle , \rangle_Q$ in this basis. 
The measure $\mu_{\nL}$ is then given by
\[ \mu_{\nL} = |A|^{1/2}\mu, \]
where $|A|$ is the modulus of the determinant. We call it the {\bf canonical measure} on $\nL$ with respect to $Q_\nL$.

Let $\nM$ be another $R$-lattice, equipped with a non-degenerate quadratic form $Q_\nM$.

Choose a basis $f_1, \dots, f_n$ of $\nM$, too, and denote
$\widetilde{\mu} := \bigwedge_{i,j} e_i^* \tensor f_j^* \in \bigwedge^{nm} \nL^* \tensor \nM^*$. 
We call $\mu_{\nL, \nM} = |A|^{n/2} |B|^{m/2} \mu$
the {\bf canonical measure} on $\nL \tensor \nM$, where $A$ are $B$ the matrices of the associated bilinear forms, 
$m=\dim(\nL)$ and $n=\dim(\nM)$.

$\{e_i\otimes e_j\}_{i \le j}$ is a basis of $\Sym^2(\nL)$. We denote the corresponding dual basis by $\{(e_i\otimes e_j)^*\}_{i \le j}$.
Let $\widetilde{\mu} = \bigwedge_{i<j} (e_i \tensor e_j)^* \in \bigwedge^{\frac{m(m+1)}{2}} \Sym^2(\nL)^* $. In this case
we call $\mu_{\nL} = |A|^{\frac{m+1}{2}} \mu$ the {\bf canonical measure} on $\Sym^2(\nL)$.

Let $\widetilde{\mu} = \bigwedge_{i} (e_i \tensor e_i)^* \wedge \bigwedge_{i<j} (e_i \tensor e_j + e_j \tensor e_i)^*$.
In this case, we call $\mu_{\nL} = |A|^{\frac{m+1}{2}} \mu$
the {\bf canonical measure} on $(\nL \tensor \nL)^s$.

Similarly, we get a {\bf canonical measure} $\mu_{\nL} = |A|^{\frac{n-1}{2}} \mu$ on $\Lambda^2 \nL$.

According to these definitions, the measures
$\mu_\nL$ on $\Sym^2(\nL)$ and $\mu_\nL$ on $(\nL^* \tensor \nL^*)^s=\Sym^2(\nL)^*$ are dual. However, 
the symmetrization map $\Sym^2(\nL^*) \stackrel{\sim}{\longrightarrow} \Sym^2(\nL)^*$ sends
the canonical measure $\mu_\nL$ of the left hand side to the $|2|^{-m}$-multiple of $\mu_\nL$ 
on the right hand side.
\end{PAR}

\begin{PAR}\label{DEFI}
Let $R$ be a ring.
Let $\nL, \nM$ be $R$-lattices of rank $m$ and $n$ with quadratic forms $Q_\nL$ and $Q_\nM$.
Assume $Q_\nL$ non-degenerate. For each $R$-algebra $R'$, we define
\[ \nIsome(\nM, \nL)(R') := \{ \alpha: \nM_{R'} \rightarrow \nL_{R'} \where \alpha  \text{ is an isometry}\}. \]
$\nIsome(\nM, \nL)$ 
is an affine algebraic variety, defined over $R$. If $Q_\nM$ is degenerate, define in addition:
\[ \nIsome^{1}(\nM, \nL)(R') := \{ \alpha: \nM_{R'} \rightarrow \nL_{R'} \where \alpha \text{ is an injective isometry } \}. \]

If $R=\Qp$ or $\Af$, for any compact open subgroup $K \subset \SO(\nL_R)$ and 
compact open subset $\kappa \subset \nL_R \otimes \nM_R^*$ there are finitely many $K$-orbits in
$\nIsome(\nM, \nL)(R) \cap \kappa$. We will write frequently
\[ \sum_{ K \alpha \subseteq \nIsome(\nM, \nL)(R)\cap \kappa} \cdots \]
meaning, that we sum over all those orbits and $\alpha$ is a respective representative. 
In the case $R=\Af$, by $\alpha^\perp_\Z$ we understand a lattice which satisfies $\alpha^\perp_\Z \otimes \Zh \cong \im(\alpha)^\perp \cap \nL_\Zh$. It can
be realized as $(\alpha')^\perp \cap \nL'_\Z$ where $\nL'_\Z$ is in the genus of $\nL_\Z$ satisfying $g \nL'_\Zh = \nL_\Zh$ for a $g \in \SO(\nL_\Af)$ with
$g\alpha' = \alpha$ for a $\alpha' \in \nIsome(\nM, \nL)(\Q)$. Such $\alpha'$ and $g$ exist because of Hasse's principle and Witt's theorem, respectively. (In
our global cases always $\nIsome(\nM, \nL)(\R)\not= \emptyset$). However, only its genus is determined and all occurring formulas/objects will depend only on it.
\end{PAR}

\begin{DEF}\label{DEFGAMMA}
\begin{eqnarray*}
\Gamma_n(s) &:=& \pi^{\frac{n(n-1)}{4}}\prod_{k=0}^{n-1}\Gamma(s-\frac{k}{2}).
\end{eqnarray*}

\begin{eqnarray*}
\Gamma_{n,m}(s) &:=& 2^n \frac{\pi^{\frac{n}{2}(s+m)}}{\Gamma_n(\frac{1}{2}(s+m))} = \prod_{k=m-n+1}^m 2 \frac{\pi^{\frac{1}{2}(s + k)}}{\Gamma(\frac{1}{2}(s+k))}.
\end{eqnarray*}
\end{DEF}
$\Gamma_n$ is the higher dimensional gamma function defined e.g. in \cite{Shimura}.

\section{Representation densities}\label{SECTIONREPDENS}

\begin{PAR}
Let $\nL_\Qp, \nM_\Qp$ be  finite dimensional $\Qp$-vector spaces with non-degenerate quadratic forms 
$Q_\nL$ and $Q_\nM$, respectively and
let $\nL_\Zp \subset \nL_\Qp, \nM_\Zp \subset \nM_\Qp$ be $\Zp$-lattices.

The main object of study of this paper are the associated representation densities, i.e. the {\bf volumes}
of $\nIsome(\nM, \nL)(\Zp)$ or more generally the integral
\[ \mu(\nL, \nM, \varphi) := \int_{\nIsome(\nM, \nL)(\Qp)} \varphi(x^*) \mu_{\nL,\nM}(x^*) \]
of a function $\varphi \in S(\nM^*\otimes \nL)$ (locally constant with compact support)
with respect to a suitable measure $\mu_{\nL,\nM}$.
The canonical measures in \ref{DEFCANVOL} induce a {\em canonical} measure 
$\mu_{\nL, \nM}$ on $\nIsome(\nM, \nL)(\Qp)$, too. In this section, we will first describe this measure (also for the real case) and then relate the volumes, respectively
integrals to the classical definition of representation density.
\end{PAR}

\begin{PAR}\label{MASSSEQUENZ}
Let $R$ be a $\Q_p$ or $\R$. Let $\nL, \nM$ be vector spaces of dimension $m$, $n$. Assume $m \ge n \ge 0$. We identify $\nM^* \tensor \nL$ with $\Hom(\nM, \nL)$ in what follows.
There is a fibration
\begin{equation}\label{fibration}
\xymatrix{ \nIsome(\nM^Q,\nL) \ar@{^{(}->}[rr] & & \nM^* \tensor \nL \ar[rr]^{\alpha \mapsto \alpha^! Q_\nL} & & \Sym^2(\nM^*), } 
\end{equation}
where $\nIsome(\nM^Q, \nL)$ is the preimage of $Q=Q_\nM$ and $\alpha^! Q_\nL$ denotes 
pullback of $Q_\nL$ to $\nM$ via $\alpha$. As soon as we choose (translation invariant) measures $\mu_{1,2}$ on $\nM^* \otimes \nL$ and $\Sym^2(\nM^*)$ respectiviely, 
this defines a measure $\frac{\mu_1}{\mu_2}$
on the fibers, restricted to the submersive set $(\nM^*_R \otimes \nL_R)^{reg}$ of the map $\alpha \mapsto \alpha^! Q_\nL$. This set coincides with the locus of maps with
maximal rank $n$. This means in particular that the following integral formula holds true:
\begin{equation}\label{integralformula}
\int_{\Sym^2(\nM_R^*)} \mu_2(Q) \int_{\nIsome(\nM^Q, \nL)(R)} \varphi(\alpha) \frac{\mu_1}{\mu_2}(\alpha)  = \int_{\nM_R^* \tensor \nL_R} \varphi(\alpha) \mu_1(\alpha),
\end{equation}
where $\varphi$ is continuous with compact support on $(\nM^*_R \otimes \nL_R)^{reg}$.
\end{PAR}

\begin{LEMMA}[{Weil \cite[\S 34]{Weil2}}]\label{LEMMAWEILINTEGRABILITY}
For $m \ge 2n+1$ and $R$ local, $\varphi \in S(\nM^*_R \otimes \nL_R)$ (space of Schwartz-Bruhat functions) is integrable with respect to the
measure $\frac{\mu_1}{\mu_2}$, too.
\end{LEMMA}

In the case $m\ge 2n+1$, the integrals $\int_{\nIsome(\nM^Q, \nL)} \varphi(\alpha) \frac{\mu_1}{\mu_2}(\alpha)$ may be computed by Fourier analysis (cf. Theorem \ref{YANG}), 
via the following well-known theorem. It is also central for the connection between Eisenstein series (or
Whittaker functions) and volumes. It is analogous to the connection between Gauss sums and
representation numbers over finite fields.

\begin{SATZ}[{Weil \cite[Proposition 6]{Weil2}}] \label{SATZWEILPROP6}
Let $R$ be local and $m \ge 2n+1$. Let $\varphi \in S(\nM^*_R \otimes \nL_R)$. The function
\[ \Psi(Q) = \int_{\nIsome(\nM^Q, \nL)} \varphi(\alpha) \frac{\mu_1}{\mu_2}( \alpha ) \qquad Q \in \Sym^2(\nM^*_R) \]
has Fourier transform
\[ \Psi'(\beta) = \int_{\nM^* \otimes \nL} \varphi(x^*) \chi((x^*)^!Q_L \cdot \beta) \mu_1 (x^*) \qquad \beta \in (\nM_R\otimes \nM_R)^s \]
with respect to a measure $\mu_2$ on $\Sym^2(\nM^*_R)$. 
Here $\mu_1$, $\mu_2$ and $\frac{\mu_1}{\mu_2}$ are connected via the fibration~\ref{MASSSEQUENZ}.
\end{SATZ}

\begin{DEF}\label{DEFCANVOLISOME}
In particular, if $Q_M$ is non-degenerate, and for any $m, n$, the canonical measures on $\nM^*_R \tensor \nL_R$ and $\Sym^2(\nM_R^*)$, introduced in \ref{DEFCANVOL},
define a {\bf canonical measure} $\mu_{\nL, \nM}$ on the fibre $\nIsome(\nM,\nL)(R)$ (which is the fibre above $Q=Q_M$) by means of this fibration.
\end{DEF}

We especially get a canonical and also invariant measure on every $\SO(\nL_R)$, coming (up to a real factor) from
an algebraic volume form. On the other
hand, algebraic invariant volume forms on $\SO(\nL_R)$ are canonically identified with $\Lambda^{\frac{m(m-1)}{2}} \Lie(\SO(\nL_R))^*$.
Hence every invariant measure on $\SO(\nL_R)$ is given by a translational invariant measure on $\Lie(\SO(\nL_R))^*$.
We have the following
\begin{LEMMA}\label{LEMMAVOLSO}
The associated measure on $\Lie(\SO(\nL_R))$ is the canonical measure (cf. \ref{DEFCANVOL}) on $\Lambda^2 \nL_R$ under the
natural identification $\Lie(\SO(\nL_R))) \cong \Lambda^2 \nL_R$ given by contraction with the bilinear form associated with $Q_\nL$.
\end{LEMMA}

The relation to the classical definition of representation density ($p$-adic case) and 
classical sphere volumes (real case) is given by the following easy
\begin{LEMMA}\label{REPDENSCLASSICAL} \quad \\
\begin{enumerate}
 \item For $n=m=1$, $\nIsome(\nM, \nL)_R$ consists of two points, each of which has volume one.
 \item For $R=\R$ and positive definite spaces $\nM_\R, \nL_\R$, we get 
\[ \vol(\nIsome(\nM, \nL)(\R)) =  \prod_{k=m-n+1}^m  2 \frac{\pi^{k/2}}{\Gamma(k/2)}.  \]
 \item For $R=\Qp$, choose lattices $\nL_\Zp$ and $\nM_\Zp$ and bases $\{f_i\}$ of $\nM_\Zp$ and $\{e_i\}$ of $\nL_\Zp$, respectively.
 Then we have:
\begin{gather*}
 \int_{\nIsome(\nM, \nL)(\Qp)} \varphi(x^*) \mu_{\nL,\nM}(x^*) = D(\nM_\Zp, \nL_\Zp) \beta(\nL_\Zp, \nM_\Zp, \varphi)
\end{gather*}
where 
\[ D(\nM_\Zp, \nL_\Zp) := |d(\nL_\Zp)|_p^{n/2} |d(\nM_\Zp)|_p^{(n-m+1)/2} \]
and
\[ \beta(\nL_\Zp, \nM_\Zp, \varphi) := \lim_{l \rightarrow \infty} p^{l(n(n+1)/2-mn)} \sum_{\substack{ \{\delta_i\} \subset
p^{-r}\nL_\Zp / p^l \nL_\Zp \\ Q_\nL(\delta_i)\equiv Q_\nM(f_i) \mod p^l \\ \langle \delta_i, \delta_j \rangle_\nL \equiv \langle f_i, f_j \rangle_\nM \mod p^l
}} \varphi(\sum_i f_i^* \otimes
\delta_i) \]
and $r$ is any integer with $\supp(\varphi) \subseteq p^{-r}\nL_\Zp$.
\end{enumerate}
\end{LEMMA}

\section{Fourier coefficients of Eisenstein series}
\label{EIS}

We briefly recall the connection between representation densities and Eisenstein series associated with the Weil representation.

\begin{PAR}
Let $R$ be $\Q_p$, $\R$, or $\A$ and $\chi$ the corresponding standard character (cf. \ref{NATURALCHARACTERS}).
Let $\nL_\Q$ be a quadratic space of dimension $m$ and $\nM_\Q$ be any finite dimensional space of dimension $n$.
Let $\Sp'(\mathfrak{\nM}_R)$ be $\Sp(\mathfrak{\nM}_R)$ if $m$ is even and $\Mp(\mathfrak{\nM}_R)$ 
(metaplectic double cover) if $m$ is odd. Here $\mathfrak{\nM} = \nM \oplus \nM^*$ with its natural symplectic form 
\[ \langle \left(\begin{matrix}w_1 \\ w^*_1\end{matrix} \right), \left(\begin{matrix}w_2 \\ w^*_2\end{matrix} \right) \rangle  \mapsto w^*_1(w_2) - w^*_2(w_1). \]

We denote:
\begin{eqnarray}
 g_l(\alpha) &:=& \left(\begin{matrix}\alpha &0\\0&^t\alpha^{-1} \end{matrix}\right)  \\
 u(\beta) &:=& \left(\begin{matrix}1 & \beta \\ 0 & 1 \end{matrix}\right) \\
 d(\gamma) &:=& \left(\begin{matrix}0 & -^t \gamma^{-1} \\ \gamma & 0 \end{matrix}\right) 
\end{eqnarray}
for $\alpha \in \Aut(\nM)$, $\beta \in \Hom(\nM^*, \nM)$ and $\gamma \in \Iso(\nM, \nM^*)$ 
and denote the image of $g_l$ and $u$ in $\Sp(\mathfrak{\nM})$ by $G_l$ and $U$ respectively.

Recall \cite{Weil1, Weil2} the Weil representation $\nweil_{\nL,\nM}$ of $\Sp'(\mathfrak{\nM}_R) \times \SO(\nL_R)$ 
on $S(\nM_R^* \otimes \nL_R)$ (space of Schwartz-Bruhat functions). It actually comes from the restriction of
the Weil representation of $\Mp(\mathfrak{\nM}_R \otimes \nL_R)$ on $S(\nM_R^* \otimes \nL_R)$.
The above elements have lifts to $\Sp'$ which act via the formul\ae:
\begin{eqnarray*}
\nweil_{\nL,\nM}(g_l(\alpha)) \varphi: &x^* \mapsto& \frac{\widetilde{\WeilGamma}(\gamma_0\alpha)}{\widetilde{\WeilGamma}(\gamma_0)}|\alpha|^{\frac{m}{2}} \varphi(^t \alpha x^*) \\
\nweil_{\nL,\nM}(u(\beta)) \varphi: &x^* \mapsto& \chi((x^*)^!Q_L \cdot \beta) \varphi(x^*) \\
\nweil_{\nL,\nM}(d(\gamma)) \varphi: &x^* \mapsto& \widetilde{\WeilGamma}(\gamma) \int_{\nM \otimes \nL^*} \varphi({}^t\gamma x) \chi(-\langle x^*, x \rangle) \mu_{\nL, \gamma}(x) 
\end{eqnarray*}
Here $\widetilde{\WeilGamma}(\gamma)$ is a certain eighth root of unity. It is equal to $\WeilGamma(\gamma \otimes Q_L)$ if $n=1$, where $\WeilGamma$ is the Weil index, denoted by $\gamma$ in \cite{Weil1, Weil2}. Warning: If $m$ is odd, the lift of $g_l(\alpha)$ does not constitute a group homomorphism.
$\mu_{\nL, \gamma}$ is the canonical measure on $\nL^* \otimes \nM$ determined by $Q_\nL$ on $\nL$ and $\gamma$ on $\nM$. 

We have obviously by the characterization of the Weil representation
\[ \Psi'(\beta) =  \widetilde{\WeilGamma}(\gamma_0)^{-1} (\nweil_{\nL,\nM}(d(\gamma_0) u(\beta)) \varphi) (0), \]  
for the function $\Psi'$ from Theorem \ref{SATZWEILPROP6}. 
Here, we assume, that the measure $\mu_1$ of that theorem has been chosen to be the canonical measure on $\nL \otimes \nM^*$ induced by $\gamma_0$ on $\nM$ and $Q_\nL$ on $\nL$.

This is the starting point for the connection of representation densities to Eisenstein series, as we will now briefly recall.
\end{PAR}

\begin{PAR}\label{EISENSTEIN} Let $R$ be $\Q_p$, $\R$, or $\A$.
Choose a maximal compact subgroup $K_R'$ of $\Sp'(R)$.
We have the Iwasawa decomposition
\[ \Sp'(\mathfrak{\nM}_R) = P' K_R'. \]
where $P'$ is the preimage of $P = U G_l$. 

Let $\xi'$ denote a character of $P'(\A)$, which is trivial on $U(\A)$ and on $P(\Q)$, 
but nontrivial on the metaplectic kernel (if $m$ is odd). If $m$ is even, $\xi$ comes from a character $\xi$ of $\A^*/\Q^*$ lifted to 
$G_l$ via $g_l(\alpha) \mapsto \xi(\det(\alpha))$. 

Let $\nInd_R(s, \xi_R')$ be the (normalized) parabolically induced representation 
\[ \nInd_R(s, \xi_R') := \nInd_{P'}^{\Sp'(\mathfrak{\nM}, R)}(|\det|^{s} \xi'_R), \] 
which is the space of smooth $K_R$-finite functions $\Psi$, satisfying
\[ \Psi( p g) = \xi_R'(p) |\det(\alpha(p))|^{s+\frac{n+1}{2}} \Psi(g), \]
where $\alpha(p)$ is the $G_l$ component in the Iwasawa decomposition of the projection of $p$ to $\Sp$.

In particular, for any $R$ as above, the Weil representation defines the following $\Sp'(\mathfrak{\nM}_R)$-equivariant operator:
\begin{eqnarray*}
 \Phi: S(\nM^*_R \otimes \nL_R) &\rightarrow& \nInd_R(\textstyle\frac{m-n-1}{2}\displaystyle, \xi'_R) \\
 \varphi &\mapsto& \{ g \mapsto (\nweil_{\nL,\nM}(g)\varphi)(0) \}
\end{eqnarray*}
Here $\xi'$ is determined by the character (up to sign, if $m$ is odd) 
$\alpha \mapsto \widetilde{\WeilGamma}(\gamma_0 \alpha)/\widetilde{\WeilGamma}(\gamma_0)$ for some $\gamma_0$.
\end{PAR}

\begin{PAR}\label{DEFEISENSTEIN}
Let $\Psi(s_0) \in \nInd(s_0, \xi')$ be given (here we set $s_0:=\frac{m-n-1}{2}$). It can be extended uniquely to a `section'
parameterized by $s \in \C$, with the property that the restriction to $K'$ is independent of $s$.

To any such `section', there is an associated {\bf Eisenstein series}. This association is a 
$\Sp'(\A)$-equivariant map
\[ E: \nInd_\A(s, \xi') \rightarrow \mathcal{A}(\Sp(\mathfrak{\nM}_\Q) \backslash \Sp'(\mathfrak{\nM}_\A)) \]
\[ \Psi(s) \mapsto \sum_{\gamma \in P(\Q) \backslash \Sp(\mathfrak{\nM}_\Q)} \Psi(s)(\gamma g), \]
where $\mathcal{A}$ is the space of automorphic functions. This series converges absolutely if $\Re(s) > \frac{n+1}{2}$ and
posesses a meromorphic continuation in $s$ to all of $\C$. Note that $\Sp(\mathfrak{\nM}_\Q)$ lifts canonically to $\Sp'(\mathfrak{\nM}_\A)$.

The Eisenstein series decomposes as follows:
\[ E(\Psi, g;  s) = \sum_{{\nM^*}' \subset \nM^*} E_{{\nM^*}'}(\Psi, g; s), \]
with
\[ E_{{\nM^*}'}(\Psi, g; s) = \sum_{\beta \in ({\nM}'\otimes{\nM}')_\Q^s} \Psi(s)(d_{M'}(\gamma_0) u(\beta) g). \]
where $d_{M'}(\gamma_0)$ is embedded via $\Sp'(\mathfrak{\nM}') \hookrightarrow \Sp'(\mathfrak{\nM})$. This embedding, 
as well as the dual subspace $\nM' \subset \nM$, depend on the choice of a complement $(\nM^*)''$. One easily sees, however, 
that $E_{{\nM^*}'}(\Psi, g; s)$ does not depend on these choices, nor on the isomorphism $\gamma_0: \nM'_\Q \rightarrow (\nM^*_\Q)'$.

At $s=s_0$, with $m > 2n+2$ (this assures convergence), we get:
\[ E_{{\nM^*}'}(\Psi, g; s_0) = \sum_{\beta \in ({\nM}'\otimes{\nM}')_\Q^s} \left(\nweil_{\nL,\nM'}(d(\gamma_0) u(\beta)) \nweil_{\nL,\nM}(g) \varphi\right)(0). \]
Using Theorem \ref{SATZWEILPROP6} and Poisson summation this yields:
\[ E_{{\nM^*}'}(\Phi(\varphi), g; s_0) = \sum_{Q \in \Sym^2({\nM^*}')_\Q} \int_{\nIsome_\A({\nM'}^{Q}, \nL)} \left(\nweil_{\nL,\nM'}(g)\varphi\right)(x^*) \dd x^*. \]
Here we interprete $g \varphi$ by composition with the embedding ${\nM^*}'\otimes \nL \hookrightarrow \nM^* \otimes \nL $
as a function on $({\nM^*}'\otimes \nL)_\A$. 
$\dd x^*$ is the Tamagawa measure.

This is essentially the Fourier expansion of the Eisenstein series: A calculation shows that
\begin{eqnarray*}
E_Q(\Phi(\varphi), g; s_0) &=& \int_{(\nM\otimes \nM)^s \backslash (\nM\otimes \nM)_\A^s} \sum_{{\nM^*}' \subset \nM^*} 
E_{{\nM^*}'}(\Phi(\varphi),u(\beta)g; s_0) \chi(\beta Q) \dd \beta \\
&=& \sum_{{\nM^*}' \subset \nM^*} \int_{(\nM' \otimes \nM')_\A^s} \left(\nweil_{\nL,\nM'}(d(\gamma_0)  u(\beta)) \nweil_{\nL,\nM}(g) \varphi \right)(0) \chi(\beta Q) \dd \beta \\
\end{eqnarray*}
\end{PAR}

It is convenient to make the following
\begin{DEF}\label{DEFWHITTAKER}
The {\bf Whittaker function} is defined as
\begin{eqnarray*}
W_{\nu, Q, {\nM^*}'}(\Psi, g) &:=& \int_{(\nM'\otimes \nM')^s_{\Q_\nu}} \Psi(\nweil_{\nL,\nM'}(d(\gamma_0) u(\beta)) \nweil_{\nL,\nM}(g)) \chi(\beta Q) \mu_{\gamma_0} (\beta_\nu), \\
\end{eqnarray*}
where we now chose the canonical measure $\mu_{\gamma_0}$ with respect to some fixed positive definite $\gamma_0$ on $\nM_\Q$ and hence $\nM_\Q'$ for convenience.
\end{DEF}

Hence we have
\begin{eqnarray*}
E_{Q}(\Psi, g; s) &=& \sum_{\{{\nM^*}' \subset \nM^* \where Q \in \Sym^2({\nM^*}') \} } \prod_\nu W_{\nu, Q, {\nM^*}'}(\Psi_\nu(s), g_\nu)
\end{eqnarray*}
and again by Theorem \ref{SATZWEILPROP6}:
\begin{equation*}\label{WHITTAKERWEIL}
 W_{\nu, Q, {\nM^*}'}(\Phi(\varphi_\nu)(s_0), g) = \widetilde{\WeilGamma}_\nu(\gamma_0) \int_{\nIsome_{\Q_\nu}({\nM'}^{Q}, \nL)} (\nweil_{\nL, \nM'}(g)\varphi_\nu) (x_\nu^*) \mu_{\nL,\gamma_0} (x_\nu^*),  
\end{equation*}
where $\mu_{\nL,\gamma_0}$ is the measure on $\nIsome_{\Q_\nu}({\nM'}^{Q}, \nL)$ induced by the canonical ones on $\Sym^2(\nM^*)$ and
$\nM^*\otimes \nL$ with respect to $\gamma_0$ and $Q_\nL$ via \ref{MASSSEQUENZ}.

More generally, the Siegel-Weil formula equals the whole Eisenstein series associated with $\Phi(\varphi)$ to an integral of a theta function
associated with $\varphi$ (see e.g. \cite{KuRa1, KuRa2, Weil2}).

Assume, that $\gamma_0$ and a lattice $\nM_\Zp$ are chosen such that $\gamma_0$ induces an isomorphism $\nM_\Zp \rightarrow \nM^*_\Zp$.
We will now investigate the Whittaker integral {\em as a function of $s$} in the case $R=\Qp$ to some extent.

\begin{SATZ}\label{REPDENSNONARCHIMEDIANINTERPOLATION}
Let $Q \in \Sym^2(\nM^*)$ be non-degenerate (with associated bilinear form $\gamma$) and $r \in \Z_{\ge 0}$. Denote $s=s_0+r$. Let $\varphi \in S(\nL_\Qp \otimes \nM_\Qp^*)$.
For sufficiently large $r$, we have the equality:
\begin{eqnarray*}
W_{Q,p}(\Phi(\varphi)(s), g) &=& \widetilde{\WeilGamma}_p(\gamma_0) \int_{\nIsome_{\Q_p}({\nM}^{Q}, \nL \perp H^r )} (\nweil_{\nL \perp H^r,\nM}(g)\varphi^{(r)})(x^*) \mu_{\gamma_0,\nL \perp H^r}(x^*), \\
&=& \widetilde{\WeilGamma}_p(\gamma_0) |\gamma|_p^{s} \mu_p(\nL, \nM^Q_\Zp, \nweil_{\nL,\nM}(g) \varphi; s-s_0)
\end{eqnarray*}
where $\varphi^{(r)}$ is $\varphi$ tensored with the characteristic function
$\chi^{(r)}$ of $H^r_{\Z_p} \otimes \nM_\Zp^*$ (depends on the choice of $\nM_\Zp$) and 
\[ \mu_p(\nL, \nM_\Zp, \varphi; r) = \mu_p(\nL \perp H^r, \nM, \varphi^{(r)}) \]
for $r \in \Z_{\ge 0}$. (Note, that the continuation of $\mu_p$ depends on the choice of the lattice $\nM_\Zp$!)

Furthermore the left hand side is a polynomial in $p^{-s}$ and therefore determined by the above values. Here $|\gamma|$ is computed with respect to the measure 
$\mu_{\gamma_0}$ on $\nM$.
\end{SATZ}

\begin{proof}
The Weil representation on 
$S(\nL^{(r)}\otimes \nM^*)$ is the tensor product of the respective Weil representation on 
$S(\nL\otimes \nM^*)$ and $S(H^{(r)}\otimes \nM^*)$.
We have
\[ (\nweil_{H^r,\nM}(u(\beta) g_l(\alpha) k ) \chi^{(r)} )(0) = |\alpha|^r,  \]
where $k \in K$ and $K$ is the maximal compact open subgroup associated with the lattice $\mathfrak{\nM}_\Zp$. Hence multiplying 
$\Phi(\varphi)(s)$ by $\Phi(\chi^{(r)})$ or substituting $s$ by $s+r$ has the same effect.

The assertion that the left hand side is a polynomial in $p^{-s}$ follows from the arguments given in \cite[p. 101]{Rallis}. We will later (cf. \ref{MUPOLYNOMIAL}) 
prove directly that $\mu_p(\dots; s)$ is a polynomial in $p^{-s}$ for sufficiently large $s$.
\end{proof}

\begin{PAR}\label{REPDENSARCHIMEDIANINTERPOLATION}

We now investigate the Whittaker integral {\em as a function of $s$} in the case $R=\R$ and for the function 
$\Psi_\infty:=\Psi_{\infty,\frac{m}{2}}$, whose restriction to $K'$ (maximal compact, see below) is $\det^{\frac{m}{2}}$.
We again consider only the case that $Q$ is non-degenerate on $\nM$.
The maximal compact subgroup $K$ of $\Sp(\mathfrak{\nM}_\R)$ is defined as follows:

Let $\gamma_0$ be a symmetric and positive definite form on $\nM$.
It defines an isomorphism:
\begin{eqnarray*}
 \nM_\C^* & \rightarrow & \mathfrak{\nM}_\R \\
 w_1 + i w_2 & \mapsto & w_1 - {^t}\gamma_0^{-1} w_2 = w_1 - \gamma_0^{-1} w_2
\end{eqnarray*}
and a corresponding map
\begin{eqnarray*}
 k: \End(\nM_\C^*) &\rightarrow& \End(\mathfrak{\nM}_\R) \\
 \alpha_1 + i \alpha_2 &\mapsto& \left( \begin{matrix}\gamma_0^{-1} \alpha_1 \gamma_0 & - \gamma_0^{-1} \alpha_2 \\ \alpha_2 \gamma_0 &
\alpha_1  \end{matrix}\right).
\end{eqnarray*}
This identifies the unitary group of the Hermitian form given by $\gamma_0$ on $\nM_\C$ with the
stabilizer of $d(\gamma_0)$ in $\Sp(\mathfrak{\nM}_\R)$, which is a maximal compact subgroup, denoted $K$. Everything depends on the choice of $\gamma_0$.
Wew define $K'$  to be the preimage of $K$ in $\Sp'(\mathfrak{\nM}_\R)$.
\end{PAR}

\begin{SATZ}\label{SATZREPDENSARCHIMEDIANINTERPOLATION}
For $\nM=\nM^*=\Q$ with $\gamma_0=1$, we get
\begin{eqnarray*}
 \lim_{\alpha\rightarrow \infty} |\alpha|^{-\frac{m}{2}} e^{2 \pi \alpha^2 Q} W_{Q, \infty}(\Psi_\infty(s), g_l(\alpha)) &=&
 \widetilde{\WeilGamma}_\infty(\gamma_0) \Gamma_{1,m}(s-s_0) |\gamma|^{s} |2 \gamma|^{\frac{1}{2}(s-s_0)}.
\end{eqnarray*}
(Here: $\gamma = 2 Q$).

For arbitrary $n$, with $m > 2n$ and $s=s_0$ (holomorphic special value), we get
\begin{eqnarray*}
|\alpha|^{-\frac{m}{2}} e^{2\pi \alpha^! Q \cdot \gamma_0^{-1}} W_{Q, \infty}(\Psi_\infty(s_0), g_l(\alpha)) &=& \begin{cases}
\widetilde{\WeilGamma}_\infty(\gamma_0) \Gamma_{n,m}(0) |\gamma|^{s_0}
& Q>0 , \\
0 & \text{otherwise,} \end{cases}
\end{eqnarray*}
where $|\gamma|$ is computed via the canonical measure $\mu_{\gamma_0}$.
\end{SATZ}

For the definition of $\Gamma_{n,m}(s)$ see~\ref{DEFGAMMA}.

\begin{BEM}
The limit in the first equation will be related to integrals of Borcherds forms in \ref{KUDLA}. 

The second equation is expected because for a positive definite space $\nL_\R$, we have
the Gaussian $\varphi_\infty \in S(\nL_\R \otimes \nM^*_\R)$, defined by
\[ \varphi_\infty(\alpha) = \exp(-2 \pi \alpha^! Q_L \cdot \gamma_0^{-1}). \]
It satisfies $\Phi(\varphi_\infty) = \Psi_{\infty,\frac{m}{2}}$, therefore we get (cf. equation (\ref{WHITTAKERWEIL}))
\begin{eqnarray*}
W_{Q, \infty}(\Psi_\infty(s_0), 1) &=& \widetilde{\WeilGamma}_\infty(\gamma_0) \int_{\nIsome(\nM^Q, \nL)(\R)} \varphi_\infty(\alpha) \mu_{\nL,\gamma_0}(\alpha) \\
&=& \widetilde{\WeilGamma}_\infty(\gamma_0) \exp(-2\pi Q \cdot \gamma_0^{-1})  \Gamma_{n,m}(0) |\gamma|^{\frac{m-n-1}{2}},
\end{eqnarray*}
in accordance with the theorem. However, below we will give a different proof of the formula, based on Shimura's work \cite{Shimura}.
\end{BEM}

\begin{proof}[Proof of theorem \ref{SATZREPDENSARCHIMEDIANINTERPOLATION}.]
The Iwasawa decomposition of the argument of $\Psi_{\infty,l}$ in the Whittaker integral can be expressed as 
\[ d(\gamma_0)u(\beta) = u(\Delta^2 \beta) g_l(\Delta) k(\gamma_0 \Delta \beta + i \gamma_0 \Delta \gamma_0^{-1}),  \]
where $\Delta = (\sqrt{1 + (\beta\gamma_0)^2})^{-1}$. It satisfies ${}^t \Delta = \gamma_0 \Delta \gamma_0^{-1}$.

Hence we get:
\[
W_{Q,\infty}(\Psi_\infty(s), 1) = \int_{(\nM\otimes \nM)^s_{\R}} |\Delta|^{s+\frac{n+1}{2}} \chi_l(\gamma_0\Delta\beta + i {}^t\Delta)
\chi(Q \cdot \beta) \mu_{\gamma_0}(\beta) \\ 
\]
and after choosing an orthonormal basis for $\gamma_0$:
\[
W_{Q,\infty}(\Psi_\infty(s), 1) = \int_{(\R^n \otimes \R^n)^s} \det(X+i)^{-a} \det(X-i)^{-b} e^{-2\pi i \tr( \frac{1}{2} X \gamma ) }
\dd X,
\]
where $\gamma$ is the bilinear form associated with $Q$ (expressed in the chosen basis in the above formula).

with $a=\frac{1}{2}(s+\frac{n+1}{2}+l)$ and $b = \frac{1}{2}(s+\frac{n+1}{2}-l)$.
Here $\det(X+i)^{-a} = e^{-a (\frac{n}{2} \pi i +	  \log(\det(1-iX)))}$ and
$\det(X-i)^{-b} = e^{-b (-\frac{n}{2} \pi i + \log(\det(1+iX)))}$,
where $\log$ is the main branch of logarithm. $\dd X$ is the measure defined in \ref{DEFCANVOL} for the standard basis, without the 
determinant factor. It is the same as used in \cite{Shimura}.

Shimura [loc. cit.] denotes this function ($\xi(1, \frac{\gamma}{2}; a, b)$ in his notation)
in analogy to the one dimensional case a {\bf confluent hypergeometric function}.
Furthermore \cite[1.29, 3.1K, 3.3]{Shimura} if $Q$ is positive definite, the RHS equals
\[ e^{-\pi\tr \gamma + i\pi\frac{n}{2}(b-a)} \pi^{na+nb} 2^n \det(\gamma)^{a+b-\frac{n+1}{2}}
\Gamma_n(a)^{-1}\Gamma_n(b)^{-1} \zeta(2\pi\gamma; a, b),  \]
where
\[\zeta(Z; a, b) = \int_{X>0} e^{-\tr Z X } \det(X+1)^{a-\frac{n+1}{2}} \det(X)^{b-\frac{n+1}{2}} \dd X. \]

In the 1 dimensional case, this gives
\[ \zeta(Z; a, b) = \Gamma(b) U(b, a+b, Z), \]
where $U(k, l, Z)$ is a solution of the classical hypergeometric differential equation
\[ Z f''(Z) + (l-Z)f'(Z) - kf(Z) = 0, \]
see \cite[\S 13]{AbSt}. We have: $U(b, a+b, Z) = Z^{-b}(1+O(|Z|^{-1}))$ [loc. cit.].

Therefore the `value at $\infty$' for $l=\frac{m}{2}$ is computed as 

\begin{eqnarray*}
&& \lim_{\alpha\rightarrow \infty} |\alpha|^{-\frac{m}{2}} e^{2 \pi \alpha^2 Q } W_{Q, \infty}(\Psi_\infty(s), g_l(\alpha)) \\
&=& \lim_{\alpha\rightarrow \infty} |\alpha|^{-s+1-\frac{m}{2}} e^{ \pi \alpha^2 \gamma} W_{\alpha^! Q, \infty}(\Psi_\infty(s), 1) \\
&=& \lim_{\alpha\rightarrow \infty} |\alpha|^{-s+1-\frac{m}{2}}
e^{-i\pi\frac{m}{4}} \pi^{s+1} 2 |\alpha^2 \gamma|^{s} \Gamma_1(a)^{-1}\Gamma_1(b)^{-1} \zeta(2\pi\alpha^2\gamma, a, b) \\
&=& e^{-i\pi\frac{m}{4}} \pi^{s+1} 2 |\gamma|^{\frac{1}{2}(s-1+\frac{m}{2})} \Gamma_1(\frac{1}{2}(s+1+\frac{m}{2}))^{-1} (2\pi)^{-\frac{1}{2}(s+1-\frac{m}{2})} \\
&=& \widetilde{\WeilGamma}_\infty(\gamma_0) \Gamma_{1,m}(s-s_0) |\gamma|^{\frac{1}{2}(s+s_0)} 2^{-\frac{1}{2}(s-s_0) },
\end{eqnarray*}
where $s_0 = \frac{m}{2}-1$.

This proves the first formula of the theorem.

For the higher dimensional case, but for $s=\frac{m-n-1}{2}$ and $l=\frac{m}{2}$, we have $a=\frac{1}{2}m$ and $b=0$ (holomorphic Eisenstein series) and we get for $m>2n$:
\begin{gather*}
\det(\alpha)^{-\frac{m}{2}}e^{ \pi\tr({^t}\alpha \gamma \alpha )} W_{Q,\infty}(\Psi_\infty(s), g_l(\alpha)) = \\
\det(\alpha)^{-\frac{m}{2}}e^{ \pi\tr({^t}\alpha \gamma \alpha )}
\det(\alpha)^{\frac{m}{2}} \int_{(\R^n\otimes \R^n)^s} \det(X+\alpha {^t}\alpha i)^{-\frac{1}{2}m} e^{- \pi i \tr (X \gamma) } \dd X \\
= e^{ \pi\tr({^t}\alpha \gamma \alpha )} e^{-i\pi\frac{nm}{4}} (2\pi)^{\frac{1}{2}nm}
\int_{(\R^n\otimes \R^n)^s} \det(2\pi i X+2 \pi \alpha {^t}\alpha)^{-\frac{1}{2}m} e^{\pi i \tr (X \gamma)} \dd X \\
= \begin{cases}
e^{-i\pi\frac{nm}{4}} \Gamma_{n,m}(0) \det(\gamma)^{\frac{m-n-1}{2} }
& \gamma>0 , \\
0 & \text{otherwise,} \end{cases}
\end{gather*}
using \cite[p. 174, (1.23)]{Shimura}.
\end{proof}

\section{The local orbit equation --- non-Archimedian case}

\begin{PAR}\label{VOLUMEFIBRATION}
Let $R$ be a $\Qp$ or $\R$.
Let $\nM$, $\nN$ and $\nL$ be $R$-vector spaces with non-degenerate quadratic forms $Q_\nM, Q_\nN, Q_\nL$ respectively.

Consider the composition map:
\[ \nIsome(\nN, \nM) \times \nIsome(\nM, \nL) \rightarrow \nIsome(\nN, \nL). \]
Fixing an $\alpha$ in $\nIsome(\nN, \nM)$, we may identify the fibre of the resulting map
\begin{eqnarray}\label{AB1}
 \nIsome(\nM, \nL) &\rightarrow& \nIsome(\nN, \nL) \\
 \delta &\mapsto& \delta \circ \alpha \nonumber
\end{eqnarray}
over $\beta \in \nIsome(\nN, \nL)$ with $\nIsome(\im(\alpha)^\perp, \im(\beta)^\perp)$.

All recursive properties that we will investigate in this section follow formally from the following theorem:
\end{PAR}

\begin{SATZ}\label{LEMMAVOLUMEFIBRATION}
The resulting fibration is compatible with the 
canonical measures. If $\dim(\nM)=\dim(\nN)$ this means that the map (\ref{AB1}) preserves volume.
\end{SATZ}
\begin{proof}
Let $Q_\nM$ be the chosen form on $\nM$. By assumption, $\alpha^!Q_\nM$ is the chosen form $Q_\nN$.

We decompose $\nM$ in $\nM_1 = \alpha(\nN)$ and $\nM_2 = \nM_1^\perp$ (orthogonal with respect to $Q_\nM$).

Decompose $\Sym^2(\nM^*)$ with respect to $Q_\nM$:
\[
\Sym^2(\nM^*) = \Sym^2(\nM_1^*)  \oplus (\nM_1^*\otimes \nM_2^*) \oplus \Sym^2(\nM_2^*),
\]
where $\Sym^2(\nM_1^*) = \Sym^2(\nN^*)$ via $\alpha$.

We get a commutative diagram of fibrations:
\[
\xymatrix{
 \nIsome(\alpha(\nN)^{\perp \gamma}, \beta(\nN)^\perp) \ar@{^{(}.>}[r] \ar@{^{(}->}[d] & \nM_2^* \otimes \nL
 \ar[rr]^-{ \delta \mapsto \M{ \delta^!Q_\nL\\ \langle \beta, \delta \rangle_{\nL} }} \ar@{^{(}->}[d]^{incl. + \beta \circ \pr_1} &&
 \Sym^2(\nM_2^*) \ar@{^{(}->}[d] \oplus (\nM_1^*\otimes \nM_2^*) \ar@{^{(}->}[d]^{incl. + \beta^! Q_\nL }\\
 \nIsome(\nM^\gamma, \nL) \ar@{^{(}->}[r] \ar[d]^{\circ \alpha} & \nM^* \otimes \nL \ar@{->>}[d]^{\circ \alpha} \ar[rr]^{\delta \mapsto \delta^!Q_\nL} && \Sym^2(\nM^*) \ar@{->>}[d]^{\alpha^!}  \\
 \nIsome(\nN^{\alpha^!\gamma}, \nL) \ar@{^{(}->}[r] & \nN^* \otimes \nL \ar[rr]^{\delta \mapsto \delta^!Q_\nL} & & \Sym^2(\nN^*)
}
\]
where $\beta$ varies in $\nN^* \otimes \nL$ such that $\beta^! Q_\nL$ varies in a neighborhood of $Q_\nN$
and $\gamma$ varies in $\Sym^2(\nM^*)$ such that $\beta^! Q_\nL = \alpha^! \gamma$.
$\nIsome(\nM^\gamma, \nL)$ is the fibre of the map $\delta \mapsto \delta^!Q_\nL$ in the middle row over $\gamma$ and 
$\nIsome(\nN^{\alpha^!\gamma}, \nL)$ is the fibre of the map $\delta \mapsto \delta^!Q_\nL$ in the bottom row above $\alpha^! \gamma$.
$\nIsome(\alpha(\nN)^{\perp \gamma}, \beta(\nN)^\perp)$ in the top-left corner is the fibre of the composition with $\alpha$ above $\beta$.

In the diagram, the vertical middle and rightmost fibrations come from (splitting) exact sequences of vector spaces.

The dotted map is defined by commutativity of the diagram. First observe that the (underlying) vertical 
exact sequences of vectorspaces are exact with respect to canonical measures on the various spaces associated with
$Q_\nM, Q_\nL$ and $\alpha^!Q_\nM=Q_\nN$ and the restrictions of $Q_\nM$ to $\nM_1, \nM_2$ respectively. The induced measure on
$\nIsome(\alpha(\nM)^{\perp \gamma}, \beta(\nN)^\perp)$ hence is described by the topmost {\em horizontal} fibration as well.

Decompose $\nM_2^* \otimes \nL = \nM_2^* \otimes (\beta(\nN) \oplus \beta(\nN)^\perp)$.
The map $\delta \mapsto \langle \beta, \delta \rangle_{Q_\nL}$ is an isomorphism $(\nM_2^* \otimes \beta(\nN)) \cong (\nM_1^* \otimes \nM_2^*)$ and
0 on the other factor. This isomorphism preserves the canonical volume, if $\beta^! Q_\nL = Q_\nN$.

Letting $\gamma$ vary only in $\Sym^2(\nM_2^*)$ fixing the other projections to 0 in $\nM_1^* \otimes \nM_2^*$ and to $Q_\nN$ in $\Sym(\nM_1^*)$ 
(i.e. having also $\beta^! Q_\nL = Q_\nN$), we get an equivalent topmost horizontal fibration:
\[
\xymatrix{
\nIsome(\alpha(\nN)^{\perp \gamma}, \beta(\nN)^\perp) \ar@{^{(}.>}[r] & \nM_2^* \otimes \beta(\nN)^\perp \ar[rr]^{ \delta \mapsto \delta^! Q_\nL } &&
 \Sym^2(\nM_2^*) }
\]
and the dotted map is equal to the canonical inclusion into $\alpha(\nN)^{\perp *} \otimes \beta(\nN)^\perp$, noting $\alpha(\nN)^{\perp *} = \nM_2^*$.
The induced measure on $\nIsome(\alpha(\nN)^{\perp \gamma}, \beta(\nN)^\perp)$, however, is {\em by definition} the canonical measure.
\end{proof}

Consider the situation of \ref{VOLUMEFIBRATION} for $R=\Qp$. 
Consider $\nN_\Qp$ via $\alpha \in \nN_\Qp^* \otimes \nM_\Qp$ as subspace of $\nM_\Qp$. Choose lattices $\nN_\Zp \subset \nM_\Zp$ such that
$\nM_\Zp = \nN_\Zp \oplus \nN_\Zp'$. Choose a third lattice $\nL_\Zp$ as well and assume that $Q_\nL \in \Sym^2(\nL^*_\Zp)$. 
Choose a coset $\kappa \in (\nL_\Zp^*/\nL_\Zp)\otimes \nM_\Zp^*$. 
Denote $\kappa = \kappa_{\nN} \oplus \kappa_{\nN'}$ the corresponding decomposition of $\kappa$.  

\begin{KOR}[Kitaoka's formula]\label{KITAOKA}
\[ \mu_p(\nL_\Zp, \nM_\Zp, \kappa) = \sum_{\substack{\SO'(\nL_\Zp) \alpha \subseteq \nIsome(\nN, \nL)(\Af) \cap \kappa_\nN \\ \kappa_{\nN'} \cap \alpha(\nN)_\Qp^\perp \otimes (\nN'_\Qp)^* \not= \emptyset }} \mu_p(\nL_\Zp, \nN_\Zp, \kappa_\nN; \alpha) \mu_p(\alpha(\nN)_\Zp^\perp, \nN_\Zp', \kappa_{\nN'}), \]
where $\kappa_{\nN'}$ is considered as an element of $((\alpha(\nN)_\Zp^\perp)^* / \alpha(\nN)_\Zp^\perp) \otimes (\nN_\Zp')^*$ via intersection with $\alpha(\nN)_\Qp^\perp \otimes (\nN_\Qp')^*$ and $\mu_p(\nL_\Zp, \nN_\Zp, \kappa_\nN; \alpha)$ is the volume of the $\SO'$-orbit of $\alpha$.
\end{KOR}
\begin{proof}
Integrate the characteristic function of 
$\kappa$ over $\nIsome(\nM, \nL)(\Qp)$
with respect to the canonical measure.

The intersections of $\kappa$ with the fibres of the (restriction) map $\nIsome(\nM, \nL)(\Qp) \rightarrow \nIsome(\nN, \nL)(\Qp)$ over $\beta \in \nIsome(\nN, \nL)(\Qp)$ can be identified with
those isometries in $\nIsome(\nN', \alpha(\nN)^\perp)(\Qp)$ which lie in $\kappa_{\nN'}$.

The volume of these sets is constant, for conjugated $\alpha$. Hence the formula follows from Theorem \ref{LEMMAVOLUMEFIBRATION}.
\end{proof}

This is a slight generalization of the case $\kappa = \nL_\Zp \otimes \nM^*_\Zp$ where we get explicitly (for the representation densities $\beta$  instead of the $\mu$'s):
\begin{KOR}[{\cite[Theorem 5.6.2]{Kitaoka}}]\label{KITAOKA2}
\begin{eqnarray*}
\beta_p(\nL_\Zp, \nM_\Zp)
&=& \sum_i
\left(\frac{d_p(K_i^\perp)}{d_p(\nN_\Zp)d_p(\nL_\Zp)}\right)^{\frac{n-k}{2}} \beta(\nL_\Zp, \nN_\Zp; K_i) \beta(K_i^\perp, \nN_\Zp'),
\end{eqnarray*}
where we numbered the orbits and wrote $K_i$ for $\alpha_i(\nN_\Zp)$ and $k$ is the dimension of $\nN$. Here one is allowed to take $\SO$-orbits, too.
\end{KOR}
\begin{proof}
The formula of the last corollary yields:
\begin{eqnarray*}
&& d_p(\nM_\Zp)^{\frac{n-m+1}{2}}
d_p(\nL_\Zp)^{\frac{n}{2}}
\beta_p(\nL_\Zp, \nM_\Zp) \\
&=&
\sum_i
d_p(\nN_\Zp)^{\frac{k-m+1}{2}}
d_p(\nL_\Zp)^{\frac{k}{2}}
\beta_p(\nL_\Zp, \nN_\Zp; K_i) \\
&\cdot&
d_p(\nN_\Zp')^{\frac{n-m+1}{2}}
d_p(K_i^\perp)^{\frac{n-k}{2}}
\beta_p(K_i^\perp, \nN'_\Zp) \\
\end{eqnarray*}
and after a reordering of the discriminant factors the statement follows.
\end{proof}

\begin{PAR}
Consider again an $\alpha \in \nIsome(\nM, \nL)$ and the resulting map:
\begin{eqnarray*}
\nIsome(\nL, \nL) &\rightarrow& \nIsome(\nM, \nL) ,\\
\delta  &\mapsto& \delta \circ \alpha.
\end{eqnarray*}
The fibre of this map over $\beta \in \nIsome(\nM, \nL)$ can again be identified with $\nIsome(\alpha(\nL)^\perp, \beta(\nL)^\perp)$.
In particular the fibre over $\alpha$ is an orthogonal group again.
We denote by $\SO(\nL_\Qp)$, respectively $\SO(\alpha(\nL_\Qp)^\perp)$ the corresponding {\em special} orthogonal groups.

Choose lattices $\nL_\Zp, \nM_\Zp$, such that $Q_\nL \in \Sym^2(\nL_\Zp)$.  Let $\kappa$ be a class in $(\nL_\Zp^*/\nL_\Zp) \otimes \nM_\Zp^*$.

Consider the discriminant kernel $\SO'(\nL_\Zp) \subseteq \SO(\nL_\Qp)$ of $\nL_\Zp$, , i.e. the kernel of the induced homomorphism
$\SO(\nL_\Zp) \rightarrow \Aut(\nL_\Zp^* / \nL_\Zp)$.

It is a compact open subgroup.
Consider the orbits of $\SO'(\nL_\Zp)$ acting on $\nIsome(\nM, \nL) \cap \kappa$.
The fibres over an orbit $\SO'(\nL_\Zp) \alpha$ all can be identified with $\SO'(\alpha_i(\nM)^\perp \cap \nL_\Zp)$.
This follows from Lemma~\ref{DISKRIMINANTENKERN} in the appendix.
\end{PAR}

\begin{KOR}[{elementary orbit equation}]\label{elementaryorbiteq}
\begin{equation*}
 \vol(\SO'(\nL_\Zp))^{-1} \mu_p(\nL_\Zp, \nM_\Zp, \kappa) =  \sum_{\SO'(\nL_\Zp)\alpha \subseteq \nIsome(\nM, \nL)(\Qp) \cap \kappa}
\vol(\SO'(\alpha(\nM)^\perp_\Zp)^{-1},
\end{equation*}
with the convention of \ref{DEFI}.
\end{KOR}
\begin{proof}
From the theorem 
 \[ \vol(\SO'(\alpha(\nM)^\perp_\Zp) \mu_p(\nL_\Zp, \nM_\Zp, \kappa; \alpha) = \vol(\SO'(\nL_\Zp)),
 \]
follows immediately, where in $\mu_p(\nL_\Zp, \nM_\Zp, \kappa; \alpha)$ we mean the volume of the orbit of $\alpha$.

Summed up over all orbits, we get the required statement.
\end{proof}

\begin{PAR}
Let $H_\Zp$ be a hyperbolic plane, $\varphi \in S(\nL_\Qp \otimes \nM_\Qp^*)$ 
and form
\[ \varphi^{(s)} := \varphi \tensor \chi_{H_\Zp^s \otimes \nM_\Zp^*} \in S((\nL_\Qp \perp H_\Qp^s) \otimes \nM_\Qp^*). \]

We are interested in the function 
\[ s \mapsto \mu_p(\nL \perp H^s, \nM_\Zp, \varphi^{(s)}), \]
for $s \in \Z_{\ge 0}$. We will denote the so constructed `interpolation' of $\mu_p(\nL_\Zp, \nM_\Zp, \kappa)$ 
by $\mu_p(\nL_\Zp, \nM_\Zp, \kappa; s)$. 
This construction is motivated by the natural continuation of Fourier coefficients of Eisenstein series, see \ref{REPDENSNONARCHIMEDIANINTERPOLATION}.
Note that the $\mu_p$ as a function of $s$ now depend on the lattice $\nM_\Zp$, too, and not only on the choice of a $\varphi \in S(\nL_\Qp \otimes \nM^*_\Qp)$. 
\end{PAR}

\begin{PAR}\label{CONTINUATIONLAMBDAMU}
We will now construct an `interpolation' of the volume of the orthogonal group, or, more precisely, of its discriminant kernel
such that the above orbit equation remains true as an identity of functions in $s$.

Assume again, $Q_\nM$, $Q_\nL$ non-degenerate. As a first step, we have
\begin{gather*}
\vol(\SO'(\nL_\Zp \oplus H_\Zp^s))^{-1} \mu_p(\nL_\Zp, \nM_\Zp, \kappa; s) \\
= \sum_{\SO'(\nL_\Zp)\alpha \subseteq \nIsome(\nM, \nL)(\Qp) \cap \kappa} \vol(\SO'(\alpha(\nM_\Zp)^\perp)^{-1}.
\end{gather*}
The stability of these orbits for arbitrary $s$ in this formula will be shown (at least for $p \not=2$) in
Lemma~\ref{STABILITY} below. This occurs at least, if $\nL$ splits $n$ hyperbolic planes. Hence the equation determines $\mu_p$ in any case
for sufficiently large $s$, if the right hand side is interpreted appropriately.
\end{PAR}

\begin{DEF}\label{DEFLAMBDAMU}
We introduce the following notation:
\begin{eqnarray*}
\lambda_p(\nL_\Zp; s) &:=& \frac{\vol(\SO'(\nL_\Zp \oplus H_\Zp^s))}{\prod_{i=1}^{s}(1-p^{-2i})}, \\
\mu_p(\nL_\Zp, \nM_\Zp, \kappa; s) &:=& \vol(\nIsome(\nM, \nL \oplus H_\Zp^s)(\Qp) \cap \kappa \oplus H_\Zp^s\otimes \nM_\Zp^*) \\
&=& |d(\nL_\Zp)|^{\frac{n}{2}} |d(\nM_\Zp)|^{-s+\frac{1+n-m}{2}}  \beta_p(\nL_\Zp, \nM_\Zp, \kappa; s),
\end{eqnarray*}
cf. also Lemma \ref{REPDENSCLASSICAL}.
Here all volumes are understood to be calculated w.r.t. the canonical measures (\ref{DEFCANVOL}).
\end{DEF}

We have the following {\bf orbit equation}:
\begin{SATZ}\label{BAHNENGLEICHUNG}
Let $p \not=2$. Assume, that $\nL_\Zp$ splits $n$ hyperbolic planes. We have
\[ \lambda_p(\nL_\Zp; s)^{-1} \cdot  \mu_p(\nL_\Zp, \nM_\Zp, \kappa; s) =  \sum_{\SO'(\nL_\Zp)\alpha \subseteq \nIsome(\nM, \nL)(\Qp) \cap \kappa} \lambda_p(\alpha(\nM_\Zp)^\perp; s)^{-1}. \]
\end{SATZ}

\begin{BEM}\label{MUPOLYNOMIAL}
We will show in \ref{EXPLIZIT} that, if $s \ge 1$, $\lambda_p(\nL_\Zp; s)$
is an actually quite simple polynomial in $p^{-s}$. It is easy to see from the explicit formula that for any sublattice $\nL'_\Zp \subseteq \nL_\Zp$, $\lambda(\nL'_\Zp; s)$
divides $\lambda(\nL_\Zp, s)$ as a polynomial in $p^{-s}$.
Hence the $\mu_p(\nL_\Zp; \nM_\Zp, \kappa; s)$ are polynomials for sufficiently large $s$ (s.t. orbits are stable), and hence 
the orbit equation makes sense for all $s \in \C$.  The fact that the $\mu_p$'s are polynomials in $p^{-s}$ for large $s$ 
can also be proven using their equality with Whittaker functions (Theorem \ref{REPDENSNONARCHIMEDIANINTERPOLATION}). 

Observe: $\prod_{i=1}^{s}(1-p^{-2i}) = (1+p^{-s}) \vol(\SO(H^s))$ for $s \in \N$ (cf. \ref{EXPLIZIT}).
\end{BEM}

In \ref{EXPLIZIT} the occurring functions will be calculated explicitely in a lot of cases. We now turn
to the problem of stability of orbits:

\begin{LEMMA}\label{LEMMA2HYPERBOLICPLANE}
Consider $\nL_\Zp \oplus H^2= \nL_\Zp \oplus \Z_p^4$ (i.e. a space with quadratic form
$Q(x_\nL,x_0,\dots,x_3) = Q_\nL(x_\nL) + x_0 x_1 + x_2 x_3$). Let $w=(w_\nL,w_0,\dots,w_3) \in \nL_\Zp \oplus H^2$ be a vector with $Q(w)\not=0$. It follows
\[ <w>^\perp \equiv H \perp \Lambda. \]
\end{LEMMA}
\begin{proof}
We may assume w.l.o.g. that $\nu_p(w_0)$ is minimal among the $\nu_p(w_i)$.

$<w>^\perp$ is described by the equation $\langle w_\nL, x_\nL \rangle + w_0x_1 + w_1x_0 + w_2x_3 + w_3 x_2 = 0$.
The map $(x_2, x_3) \mapsto (0_\nL, 0, -\frac{w_2 x_3 + w_3 x_2}{w_0}, x_2, x_3)$ therefore is an
isometric embedding of an hyperbolic plane into $<w>^\perp$. The assertion now follows from Lemma \ref{KITAOKATSSS}.
\end{proof}

Let $\kappa \in (\nL_\Zp^*/ \nL_\Zp) \otimes \nM^*_\Zp $ and $\{\alpha_i\}$ be a set of  
representatives of the orbits under $\SO'(\nL_\Zp)$ acting on $\nIsome(\nM, \nL)(\Qp) \cap \kappa$.

\begin{LEMMA}[stability of orbits]
\label{STABILITY}
Assume $\nM_\Zp$ has dimension $n$, $p\not=2$ and $\nL_\Zp$ splits $n$ hyperbolic planes.

Then $\{\alpha_i\}$ is a set of representatives of the $\SO'(\nL_\Zp \oplus H_\Zp^s)$-orbits 
in $\nIsome(\nM, \nL \perp H^s)(\Qp) \cap (\kappa \oplus (H_\Zp^s \otimes \nM_\Zp^*))$ for all $s$.
\end{LEMMA}
\begin{proof}
We begin by showing that, if $\alpha_i = g \alpha_j$ for some $g \in \SO'(\nL_\Zp \oplus H_\Zp^s)$, then
we have $\alpha_i = g' \alpha_j$ for some $g' \in \SO'(\nL_\Zp)$ as well.
We have
\[ \alpha_i(\nM_\Zp)^\perp = H_\Zp^s \perp \alpha_i(\nM_\Zp)^{\perp \nL_\Zp}. \]
Since the form is integral in $\nL_\Zp$, we have according to Lemma~\ref{KITAOKAT522}:
\[ \alpha_i(\nM_\Zp)^\perp = g(H_\Zp^s) \perp \Lambda_\Zp \]
(because $H_\Zp^s \perp \alpha_j(\nM_\Zp)$.)
Hence (Lemma~\ref{KITAOKAT522} --- here $p\not=2$ is used), there is an isometry
$\alpha_i(\nM_\Zp)^\perp$, which maps $g(H_\Zp^s)$ to $H_\Zp^s$ and lies in
$\SO'(\alpha_i(\nM_\Zp)^\perp)$. Hence it lifts to an isometry in $\SO'(\nL_\Zp \oplus H_\Zp^s)$,
which fixes $\nM_\Zp$ pointwise (Lemma~\ref{DISKRIMINANTENKERN}). Composition with $g$ yields the required $g'$.

Secondly, let an isometry $\alpha: \nM_\Zp \rightarrow \nL_\Zp \perp H_\Zp^s$ be given. We have to show that it is maped by an element in
$\SO'(\nL_\Zp \oplus H_\Zp^s)$ to any of the $\alpha_i$. It clearly suffices (induction on $s$) to consider the case $s=1$.

We proceed by induction on $n$ and first prove the case $n=1$:
By Lemma~\ref{LEMMA2HYPERBOLICPLANE} $\alpha_i(\nM_\Zp)^\perp$ ($\perp$ with respect to $\nL_\Zp \oplus H_\Zp$) splits an hyperbolic plane because
by assumption $\nL_\Zp$ splits already one.
Then apply Lemma~\ref{KITAOKAT522}. We get
\[ \alpha(\nM_\Zp)^\perp = \Lambda_\Zp \perp \Lambda'_\Zp, \]
with $\Lambda_\Zp \cong H_\Zp$.

In addition, we have (again Lemma~\ref{KITAOKAT522})
\[ \nL_\Zp \oplus H_\Zp = \Lambda_\Zp \perp \Lambda_\Zp^\perp, \]
hence (Lemma~\ref{KITAOKAC531}) $\Lambda_\Zp^\perp \cong \nL_\Zp$ and the image of $\alpha$ of course lies in
$\Lambda_\Qp^\perp$ because $\Lambda_\Qp \perp \alpha(\nM_\Qp)$. Now there is an isometry $g$, which maps
$\Lambda^\perp_\Zp$ to $\nL_\Zp$ and $\Lambda_\Zp$ to $H_\Zp$ 
(even in $\SO'(\nL_\Zp \oplus H_\Zp)$ --- Lemma~\ref{KITAOKAC531}).
The image of $g \circ \alpha$ then lies in $\nL_\Qp$ and the element of $\SO'(\nL_\Zp)$, which maps $g \circ \alpha$ to any
$\alpha_i$, lifts to an isometry $g' \in \SO'(\nL_\Zp \oplus H_\Zp)$. Hence $\alpha$ is conjugated to $\alpha_i$ under $\SO'(\nL_\Zp \oplus H_\Zp)$.

We now assume, that the statement has been proven for $\nM$ up to dimension $n-1$. 
We choose some splitting $\nM = \nN \oplus \nN^\perp$ with $\dim(\nN)=n-1$
and accordingly decomposition $\kappa = \kappa_\nN + \kappa_{\nN^\perp}$.
Let $\alpha: \nM_\Zp \rightarrow \nL_\Zp \perp H_\Zp^s$ be given. Induction hypothesis shows, that w.l.o.g. $\alpha(\nN) \subset \nL_\Qp$.
Since $\nL_\Zp$ splits an unrequired hyperbolic plane, we may even assume, that $\alpha(\nN)^\perp \cap \nL_\Zp$ splits a hyperbolic plane, too.
Hence we may apply the $n=1$ case to $\alpha(\nN)^\perp \oplus H_\Zp^s$ and $\nM^\perp$ and observe that the constructed isometries in this step all lift by Lemma~\ref{DISKRIMINANTENKERN}.
\end{proof}

We also get immediately an interpolated version of Kitaoka's formula:

\begin{SATZ}\label{KITAOKAINTERPOLATED}
Let $p \not=2$. Assume, that $\nL_\Zp$ splits $n$ hyperbolic planes. With the notation of Corollary \ref{KITAOKA}, we have
\begin{gather*} 
\mu_p(\nL_\Zp, \nM_\Zp, \kappa; s) = \\ 
\sum_{\SO'(\nL_\Zp)\alpha \subseteq \nIsome(\nM, \nL)(\Qp) \cap \kappa} \mu_p(\nL_\Zp, \nN_\Zp, \kappa_\nN, \alpha; s) 
\mu_p(\alpha(\nN)_\Zp^\perp, \nN_\Zp', \kappa_{\nN'}; s).
\end{gather*}
\end{SATZ}

The quantities $\mu_p(\nL_\Zp, \nN_\Zp, \kappa_\nN, \alpha; s)$ are by Theorem \ref{LEMMAVOLUMEFIBRATION} equal to
\[ \frac{\lambda_p(\nL_\Zp; s)}{\lambda_p(\alpha(\nN)_\Zp^\perp; s)}, \]
hence they are polynomials in $p^{-s}$ (cf. \ref{MUPOLYNOMIAL}) and furthermore, the equation can be rewritten in the more symmetric form: 
\begin{gather*}
\lambda_p(\nL_\Zp; s)^{-1} \mu_p(\nL_\Zp, \nM_\Zp, \kappa; s) = \\
\sum_{\SO'(\nL_\Zp)\alpha \subseteq \nIsome(\nM, \nL)(\Qp) \cap \kappa} \lambda_p(\alpha(\nN)_\Zp^\perp; s)^{-1} \mu_p(\alpha(\nN)_\Zp^\perp, \nN_\Zp', \kappa_{\nN'}; s).
\end{gather*}

\section{The local orbit equation --- Archimedian case}

\begin{DEF}\label{DEFLAMBDAMUARCHIMEDISCH}
We will define factors at $\infty$  analogously to $\lambda_p$ and $\mu_p$ (\ref{DEFLAMBDAMU}).
\begin{eqnarray*}
 \lambda_\infty(\nL; s) &:=& \Gamma_{m-1,m}(s), \\
 \mu_\infty(\nL, \nM; s) &:=& \Gamma_{n,m}(s),
\end{eqnarray*}
where, as usual, $n=\dim(\nM)$ and $m=\dim(\nL)$.
\end{DEF}

With this notation, we have also a (rather trivial) Archimedian analogue of the orbit equation (with only one orbit):
\begin{SATZ}\label{BAHNENGLEICHUNGARCHIMEDISCH}
If $Q_\nM$ is positive definite, we have
\[ \mu_\infty(\nL_\Z, \nM_\Z, \kappa; s) \cdot \lambda_\infty(\nL_\Z; s)^{-1} = \lambda_\infty(\alpha(\nL)^\perp; s)^{-1}. \]
\end{SATZ}
Here $\alpha$ is any real embedding $\nM_\R \rightarrow \nL_\R$, with $\alpha^!Q_\nL = Q_\nM$.
The above depends only on the respective dimensions and is formulated in dependence of $\nL$ and $\nM$ only 
in order to have the same shape than the non-Archimedian orbit equation (Theorem~\ref{BAHNENGLEICHUNG}).

Furthermore in the positive definite case, we have, analogously to the non-Archimendian case:
\[ \vol(\nIsome(\nM, \nL)_\R) = \Gamma_{n,m}(0)  \qquad n \ge m \]
and in particular
\[ \vol(\SO(\nM_\R)) = \Gamma_{m-1,m}(0). \]
(for the canonical volumes~\ref{DEFCANVOL}), see Lemma \ref{REPDENSCLASSICAL}.

\section{Explicit formul\ae, $n=1$ case}\label{SECTIONN1}

In the case $n=1$, the representation densities have been computed explicitely by Yang \cite{Yang1}:

\begin{PAR}
Assume $p\not=2$. Consider
\[ \nL_\Zp = <\varepsilon_1 p^{l_1}, \dots, \varepsilon_m p^{l_m}>, \] 
where $\varepsilon_i \in \Z_p^*$ and
$l_i \in \Z_{\ge 0}, l_1 \le \cdots \le l_m$. Assume, that $p^{-1}Q_\nL$ is not integral, i.e. we have $l_1=0$.

Denote:
\begin{eqnarray*}
 L(k, 1)  &:=& \{1 \le i \le m \where l_i-k <0 \text{ is odd}\} \\
 l(k, 1)  &:=& \#L(k,1) \\
 d(k)     &:=& k + \frac{1}{2}\sum_{l_i < k} (l_i-k) \\
 v(k)     &:=& \left(\frac{-1}{p}\right)^{ \lfloor \frac{l(k,1)}{2} \rfloor} \prod_{i \in L(k,1)} \left(\frac{\varepsilon_i}{p}\right).
\end{eqnarray*}
\end{PAR}

\begin{SATZ}[{\cite[Theorem 3.1]{Yang1}}]
\label{YANG}
With this notation, we have
\[
 \beta_p(\nL_\Zp, <\alpha p^a>, \nL_\Zp; s) = 1 + R(\alpha p^a; p^{-s}), \]
where $\alpha \in \Z_p^*$, $a \in \Z_{\ge 0}$ and
\begin{eqnarray*}
 R(\alpha p^a; X) &=& (1-p^{-1})\sum_{\substack{0<k\le a \\ l(k,1) \text{ is even}}} v(k)p^{d(k)}X^k \\
 &+& v(a+1)p^{d(a+1)}X^{a+1}\begin{cases}
  -p^{-1} & \text{if } l(a+1, 1) \text{ is even,} \\
  \left(\frac{\alpha}{p}\right)p^{-\frac{1}{2}} & \text{if } l(a+1, 1) \text{ is odd.}
\end{cases}
\end{eqnarray*}
We have:
\[
 \beta_p(\nL_\Zp, <0>, \nL_\Zp; s) = 1 + R(0; p^{-s}), \]
where
\[ R(0; X) = (1-p^{-1})\sum_{\substack{k>0 \\ l(k,1) \text{ is even}}} v(k)p^{d(k)}X^k. \]
\end{SATZ}

Furthermore, still for $n=1$, the representation densities have the following interpretations:

\begin{LEMMA}\label{REPDENS}
Let $s \in \Z_{\ge0}$ and let $\kappa$ be a coset in $\nL^*_\Zp/\nL_\Zp$.

We have the following relation to representation numbers:
\begin{equation}\label{SELTSAMEFORMEL}
\beta_p(\nL_\Zp, <q>, \kappa; s) = \#\Omega_{\kappa, q}(w) p^{w(1-m-s)}+(1-p^{-s})\sum_{j=0}^{w-1}\#\Omega_{\kappa, q}(j)p^{j(1-m-s)}
\end{equation}
Here $\Omega_{\kappa, q}(j)=\{v \in \nL_{\Z/p^j\Z} \where Q_\nM(v+\kappa) \equiv q \enspace (p^j) \}$ and $w$ is a sufficiently large integer.
(Explicitely: $w\ge 1+2\nu_p(2q \ord(\kappa) )$ --- the formula then does not depend on $w$.)

This can be written as follows $(\re s > 1)$:
\begin{equation}\label{SELTSAMEFORMEL2}
 \sum_l \frac{\#\Omega_{\kappa, q}(l)}{p^{l(m-1+s)}} = \frac{\beta_p(\nL_\Zp, <q>, \kappa; s)}{1-p^{-s}}. 
\end{equation}

If $m \ge 2$,
\begin{equation}\label{SELTSAMEFORMEL3}
  \int_{\kappa} |Q_\nL(v)-Q|^s \dd v = p^s + \beta(\nL_\Zp,<Q>,\kappa; s+1)\frac{1-p^{s}}{1-p^{-s-1}},  
\end{equation}
where $\dd v$ is the translation invariant measure with $\vol(\nL_\Zp)=1$.
\end{LEMMA}

\begin{proof}
Formula (\ref{SELTSAMEFORMEL}) is obviously true for $s=0$. Under the substitution $\nL \leadsto \nL \oplus H$ the left hand side becomes
$\beta_p(\nL_\Zp, <q>, \kappa; s+1)$ and the right hand side becomes the same expression for $s+1$, if we use the relation:
\begin{equation}\label{relation}
\# \Omega_{\kappa,q}(\nL \oplus H, r) = \sum_{\nu=0}^{r-1} p^{(r-\nu)m} (p^r-p^{r-1}) \#\Omega_{\kappa,q}(\nL, \nu) + \#\Omega_{\kappa,q}(\nL, r)p^r.
\end{equation}

Proof of the relation: An explicit calculation shows:
\begin{equation}
\#\Omega_n(H, l) = \begin{cases} (\nu_p(n)+1)(p^l - p^{l-1}) & \nu_p(n)<l, \\ l(p^l - p^{l-1})+ p^l & \nu_p(n	)\ge l. \end{cases}
\end{equation}
Hence:
\begin{eqnarray*}
&& \#\Omega_{\kappa, q}(\nL+H, l) \\
&=& \sum_{n \in \Z/p^l\Z} \#\Omega_{\kappa, q-n}(\nL, l) \#\Omega_{n}(H, l) \\
&=& \sum_{\nu=0}^l \#\Omega_{p^\nu}(H, l) \sum_{n \in \Z/p^l\Z \atop \nu_p(n)=\nu} \#\Omega_{\kappa, q-n}(\nL, l) \\
&=& \sum_{\nu=0}^l \#\Omega_{p^\nu}(H, l) \left( \sum_{n \in \Z/p^l\Z \atop \nu_p(n)\ge\nu} \#\Omega_{\kappa, q-n}(\nL, l) - \sum_{n \in \Z/p^l\Z \atop \nu_p(n)\ge \nu+1} \#\Omega_{\kappa, q-n}(\nL, l) \right)\\
&=& \sum_{\nu=0}^l \#\Omega_{p^\nu}(H, l) \left( p^{(l-\nu)m} \#\Omega_{\kappa, q}(\nL, \nu) - p^{(l-\nu-1)m} \#\Omega_{\kappa, q}(\nL, \nu+1) \right).
\end{eqnarray*}
From this the relation (\ref{relation}) follows. (\ref{SELTSAMEFORMEL2}) is obtained by letting $w \rightarrow \infty$ since
(\ref{SELTSAMEFORMEL}) does not depend on $w$.

For formula (\ref{SELTSAMEFORMEL3}) observe that
\begin{eqnarray*}
\int_{\kappa} |Q_\nL(v)-q|^s \dd v &=& \sum_{i=0}^\infty \left( \vol\kappa \cap \{|Q_\nL-Q|\le\frac{1}{p^i}\} - \vol \kappa \cap \{|Q_\nL-Q|\le\frac{1}{p^{i+1}}\}\right) \frac{1}{p^{is}} \\
&=& 1 + \sum_{i=1}^\infty \vol \kappa \cap \{|Q_\nL-Q|\le\frac{1}{p^i}\}(1-p^s)\frac{1}{p^{is}} \\
&=& 1 + \sum_{i=1}^\infty \frac{\#\Omega_{\kappa,q}(i)}{p^{i(s+n)}}(1-p^{s}) \\
&=& 1 - (1-p^s) + \sum_{i=0}^\infty \frac{\#\Omega_{\kappa,q}(i)}{p^{i(s+n)}}(1-p^{s}).
\end{eqnarray*}
From this (\ref{SELTSAMEFORMEL3}) follows using identity (\ref{SELTSAMEFORMEL2}).
\end{proof}

\begin{PAR}
We will investigate the zeta function representation given in the theorem a bit further. It is convenient to write
\[ \beta(\nL_\Zp,<0>; s) = 1 + (1-p^{-1})\delta(p^{-s}),   \]
where $\delta$ is a rational function, which has, according to Theorem~\ref{YANG}, the expansion:
\[ \delta(X) = \sum_{k>1 \atop l(k,1)\equiv 0\enspace(2)} \nu(k)p^{d(k)}X^k \]
(with the local notation from Theorem~\ref{YANG}). Here $\delta(0)=0$.

Therefore:
\begin{eqnarray*}
  E := \vol\{x \in \nL_\Zp | Q_\nL(v) \in \Z_p^* \} &=& \lim_{s \rightarrow \infty}
  \int_{\nL_\Zp} |Q_\nL(v)|^s \dd v \\
  &=& \lim_{X \rightarrow 0} (1 - \frac{1+(1-p^{-1})\delta(p^{-1}X)}{1-p^{-1}X}) \\
  &=& (1-p^{-1})(1-\delta'(0)p^{-1}).
\end{eqnarray*}
with
\[ \delta'(0) = \begin{cases} \nu(1)p^{d(1)} & l(1,1)\equiv 0 \enspace(2), \\ 0 & l(1,1)\equiv 1 \enspace(2). \end{cases} \]
\end{PAR}

\begin{DEF}\label{localzeta}
We define the {\bf normalized local zeta function} associated with $\nL$ by
\[ \zeta_p(\nL_\Zp; s) := \frac{1}{E} \int_{\nL_\Zp} |Q_\nL(v)|^{s-1} \dd v. \]
\end{DEF}

For two dimensional lattices, this coincides for example with the usual zeta function of the associated order
in the associated quadratic field ($\frac{\dd v}{E|Q_\nL|}$ is the multiplicatively invariant measure 
for which $(\OOO \tensor \Zp)^*$ has volume 1).

\begin{PAR}\label{ZETAEXPLIZIT}
Here, we explicitly compute the zeta function for an arbitrary two dimensional lattice. 
This will be used in section \ref{MODULARCURVE}.

Let $\nL_\Zp$, $p \not=2$ be a lattice with $Q_\nL \in \Sym^2(\nL_\Zp^*)$, such that $p^{-1}Q_\nL$ is not integral.
The zeta function of such a 2 dimensional lattice depends only on the discriminant because it is invariant multiplication of the form
by a scalar $\in \Z_p^*$. Hence we may assume w.l.o.g. $\nL_\Zp = \Z_p^2$ with  $Q_\nL: x \mapsto x_1)^2 + \varepsilon p^l (x_2)^2$.

With the notation of Yang (Theorem~\ref{YANG}), we have:

\begin{eqnarray*}
L(k;1) &=& \begin{matrix}
l \text{ even} & & l \text{ odd} \\
\{\}    & 0<k\le l \text{ even   } & \{\} \\
\{1\}   & 0<k\le l \text{ odd } & \{1\} \\
\{\}    & l<k   \text{ even   } & \{2\} \\
\{1,2\} & l<k   \text{ odd } & \{1\}
\end{matrix} \\
d(k) &=& \begin{cases} \frac{1}{2}k & k \le l, \\ \frac{1}{2}l & l < k, \end{cases} \\
\nu(k) &=& \begin{cases} (\frac{-\varepsilon}{p})^k & l<k, l \text{ even}, \\
1 & k \le l \text{ or } l \text{ odd }. \end{cases}
\end{eqnarray*}
\end{PAR}

Assume first $l=0$, then (as expected):
\begin{eqnarray*}
 \delta(X) &=& \sum_{k > 1} (\frac{-\varepsilon}{p})^k X^k = \frac{(\frac{-\varepsilon}{p})X}{1-(\frac{-\varepsilon}{p})X}, \\
 \zeta_p(\nL_\Zp; s) &=& \frac{(1-p^{-1})(1+(1-(pX)^{-1})\delta(X))}{(1-X)E} =
 \frac{1}{(1-X)(1-(\frac{-\varepsilon}{p})X)}.
\end{eqnarray*}
For $l$ odd, the above yields:
\begin{eqnarray*}
\delta(X) &=& \sum_{k'=1}^\frac{l-1}{2} p^{k'}X^{2k'} = (pX)^2 \frac{1-(pX^2)^{\frac{l-1}{2}}}{1-(pX)^2}, \\
\zeta_p(\nL_\Zp; s) &=& \frac{1-(pX^2)^{\frac{l+1}{2}}-X+X(pX^2)^{\frac{l-1}{2}} }{(1-X)(1-pX^2)}.
\end{eqnarray*}
For $l\ge2$, even, it yields:
\begin{eqnarray*}
\delta(X) &=& \sum_{k=l}^\infty (\frac{-\varepsilon}{p})^k p^{\frac{1}{2}l} X^k + \sum_{k'=1}^{\frac{l}{2}-1}p^{k'}X^{2k'} \\
&=& p^{\frac{1}{2}l}X^l \frac{1}{1-(\frac{-\varepsilon}{p})X} + pX^2 \frac{1-(pX^2)^{\frac{l}{2}-1}}{1-pX^2},  \\
\zeta_p(\nL_\Zp, s) &=& \frac{p^{\frac{l}{2}}X^l -p^{\frac{l}{2}-1}X^{l-1} }{(1-X)(1-(\frac{-\varepsilon}{p})X)} +
\frac{1-(pX^2)^{\frac{l}{2}}-X+X(pX^2)^{\frac{l}{2}-1} }{(1-X)(1-pX^2)}.
\end{eqnarray*}

\section{Explicit formul\ae, general case}

In this section, we will compute the functions $\mu$ and $\lambda$ (cf. Definiton~\ref{DEFLAMBDAMU}) explicitly in special cases. The expression
for $\lambda$ is quite general and can in principle be used to compute it for all lattices.

\begin{SATZ}\label{EXPLIZIT}Assume $p\not=2$.
\begin{enumerate}
\item Let $s \in \N$. $\vol(\SO(H^s)) = (1-p^{-s}) \cdot \prod_{i=1}^{s-1} (1-p^{-2i})$.
\item Let $s \in \Z_{\ge 0}$. $\lambda(\nL_\Zp\perp H; s) = (1-p^{-2s-2}) \cdot \lambda(\nL_\Zp; s+1)$.
\item Let $s \in \Z_{\ge 0}$. Let $\nL_\Zp=<\varepsilon_1, \dots, \varepsilon_k> \perp \nL'_\Zp$, 
where $\varepsilon_i \in \Z_p^*$ and $p^{-1}Q_L$ is integral on $\nL'_\Zp$. Assume $k>1$. 
Let $\varepsilon := (-1)^{\frac{k}{2}}\prod_{i=1}^{k} \varepsilon_i$, if $k$ is even. Then we have
\[ \frac{\lambda(\nL_\Zp; s)}{\lambda(\nL_\Zp; 0)} = |D|_p^{s} \prod_{i=1}^{\lfloor\frac{k-1}{2}\rfloor} \frac{1-p^{-2i-2s}}{1-p^{-2i}} \begin{cases}
1 & k \equiv 1 \enspace(2), \\ \frac{1 - (\frac{\varepsilon}{p})p^{-\frac{k}{2}-s}}{1 - (\frac{\varepsilon}{p})p^{-\frac{k}{2}}} & k \equiv 0 \enspace(2).
\end{cases}
\]
In particular $\lambda(\nL_\Zp; s)$ is a (quite simple) polynomial in $p^{-s}$. 
\item 
Let $s \in \Z_{\ge 0}$. For a unimodular lattice of discriminant $2^{m}\varepsilon$ and $\varepsilon' \in \Zp^*$ we have:
\[
\mu(\nL_\Zp, <\varepsilon'>; s) = \begin{cases} (1-(\frac{(-1)^{\frac{m}{2}}\varepsilon}{p})p^{-\frac{m}{2}-s}) & m \equiv 0 \enspace(2),
\\ (1+(\frac{(-1)^{\frac{m-1}{2}}\varepsilon\varepsilon'}{p})p^{-\frac{m-1}{2}-s}) & m \equiv 1 \enspace(2). \end{cases}
\]
\item
Let $s \in \Z_{\ge 0}$. For a unimodular lattice of discriminant $2^{m}\varepsilon$ we have:
\[ \lambda(\nL_\Zp; s) = \prod_{i=1}^{\lfloor\frac{m-1}{2}\rfloor}(1-p^{-2i-2s})\begin{cases}(1-(\frac{(-1)^{\frac{m}{2}}\varepsilon}{p})p^{-\frac{m}{2}-s}) &
m \equiv 0 \enspace(2), \\ 1 & m \equiv 1 \enspace (2). \end{cases} \]

\item 
Let $s \in \Z_{\ge 0}$. For a lattice with $\nL^*_\Zp/\nL_\Zp$ cyclic 
of order $p^\nu \not= 1$ and dimension $m\ge 2$, we may assume $\nL=\Z_p^m$, $Q_L(x) = \sum_{j=1}^{m-1} \varepsilon_j x_j^2 + p^\nu \varepsilon_m x_m^2$. Denote 
$\varepsilon = (-1)^{\frac{m-1}{2}} \prod_{j=1}^{m-1} \varepsilon_j$ if $m$ is odd. We have
\[ \lambda(\nL_\Zp; s) = |p^\nu|_p^{s+\frac{m-1}{2}} \prod_{i=1}^{\lfloor\frac{m}{2}\rfloor-1} (1-p^{-2i-2s}) \begin{cases}
1 & m \equiv 0 \enspace(2), \\ 1 - (\frac{\varepsilon}{p})p^{-\frac{m-1}{2}-s} & m \equiv 1 \enspace(2).
\end{cases}
 \]
\end{enumerate}
\end{SATZ}
\begin{proof} 
1. According to Kitaoka's formula (cf. Theorem~\ref{KITAOKA}), we have
\[\vol(\SO(H^s)) = \beta_p(H^s, 1) \beta_p(H^{s-1}\perp<-1>, -1) \vol(\SO(H^{s-1})).  \]
Theorem~\ref{YANG} yields $\beta_p(H^{s}, 1) = 1 - p^{-s}$ and $\beta_p(H^{s-1}+<-1>, 1) = 1 + p^{-s+1}$.
Furthermore, we have $\vol(\SO(H))=\beta_p(H, 1)=1-p^{-1}$.

2. follows immediately from the definition of $\lambda$ and 1.

3. Let $\nL_\Zp$ be a lattice, and $S=<\alpha_1,\alpha_2>$ a unimodular plane (e.g. a hyperbolic one).
Using the (elementary) orbit equation, Theorem~\ref{elementaryorbiteq} and Theorem~\ref{KITAOKA}, we get
\begin{eqnarray} \label{eqsurf}
\vol(\SO'(\nL_\Zp \perp S)) &=& |D|^{\frac{1}{2}} \beta_p(\nL_\Zp \perp S, \alpha_2) |D|^{\frac{1}{2}} \beta_p(\nL_\Zp \perp <\alpha_1>, \alpha_1) \nonumber \\
&& \cdot \vol(\SO'(\nL_\Zp)).
\end{eqnarray}

Hence, we have to apply Theorem~\ref{YANG} to forms of the shape
\[\nL_\Zp' = <\varepsilon_1,\dots,\varepsilon_{k'},p^{\nu_{k'+1}}\varepsilon_{k'+1},\dots,p^{\nu_{m}}\varepsilon_{m}>.\]
We get
\[ \beta_p(\nL_\Zp', \varepsilon'; s) = 1 + v(1)p^{d(1)}p^{-s}\begin{cases}
-p^{-1} & l(1,1) \equiv 0 \enspace(2), \\
(\frac{-\varepsilon'}{p})p^{-\frac{1}{2}} & l(1,1) \equiv 1 \enspace(2).
\end{cases} \]
We have
\begin{eqnarray*}
 l(1,1) &=& k' \\
 d(1)   &=& 1-\frac{1}{2}k' \\
 v(1)   &=& (\frac{-1}{p})^{\lfloor\frac{k'}{2}\rfloor} \prod_{i=1}^{k'} (\frac{\varepsilon_i}{p})
\end{eqnarray*}
Hence
\[ \beta_p(\nL_\Zp', \varepsilon'; s) = \begin{cases}
1 - (\frac{(-1)^\frac{k'}{2}\prod_{i=1}^{k'}\varepsilon_i}{p}) p^{-\frac{k'}{2}-s} & k' \equiv 0 \enspace (2), \\
1 + (\frac{(-1)^\frac{k'-1}{2}\prod_{i=1}^{k'}\varepsilon_i\varepsilon'}{p}) p^{-\frac{k'-1}{2}-s} & k' \equiv 1 \enspace (2).
\end{cases}
\]
Applying this to $\nL_\Zp'=\nL_\Zp \perp S$ and $\nL_\Zp' \perp <\alpha>$, we get the result for $k$ odd. For $k$ even write $\nL_\Zp = \nL_\Zp' \perp <\alpha>$, use
\begin{eqnarray}\label{eqvect}
\vol(\SO'(\nL_\Zp')) &=& |D|^{\frac{1}{2}} \beta_p(\nL_\Zp', \alpha) \cdot \vol(\SO'(\nL_\Zp)).
\end{eqnarray}
twice and the $k$ odd part.
Recall (Lemma \ref{KITAOKAC531}) that vectors of length $\alpha \in \Z_p^*$ form one orbit under $\SO'$, as long as the lattice in question splits a 
unimodular plane, otherwise there are 2 orbits of equal volume.

4. This is Siegel's formula, a special case of Theorem~\ref{YANG}.

5. Follows from 4. and the orbit equation, Theorem~\ref{BAHNENGLEICHUNG}.

6. Follows from 3. and the following calculation for $m \ge 2$ (which follows easily from equations (\ref{eqsurf}) and (\ref{eqvect}) and the fact $\vol(\SO'(<x>))=1$)). Observe, that there are 2 orbits (of equal volume) of vectors of length $\beta$ in a lattice $<\alpha p^\nu, \beta>$. 
\[ \lambda(\nL_\Zp;0) = |p^\nu|^{\frac{m-1}{2}} \prod_{i=1}^{\lfloor \frac{m}{2} \rfloor - 1} (1-p^{-2i}) \begin{cases} 1 & m \equiv 0 \enspace (2), \\ 1 - \left(\frac{\varepsilon}{p}\right)p^{-\frac{m-1}{2}} & m \equiv 1 \enspace (2). \end{cases} \] 
Here $\varepsilon = (-1)^{\frac{m-1}{2}} \prod_{i=1}^{m-1} \varepsilon_i$.
\end{proof}

Without proof, we give here come calculations in the case $p=2$. We plan to
extend all results of this paper to the case $p=2$ soon.

\begin{SATZ} \label{EXPLIZIT2}\quad\\
\begin{enumerate}
\item Let $s \in \N$. $\vol(\SO(H^s)) = (1-2^{-s}) \cdot \prod_{i=1}^{s-1} (1-2^{-2i})$.
\item Let $s \in \Z_{\ge 0}$. $\lambda(\nL_{\Z_2}\perp H; s) = (1-2^{-2s-2}) \cdot \lambda(\nL_{\Z_2}; s+1)$.
\item Let $s \in \Z_{\ge 0}$. For a unimodular lattice of even dimension of discriminant $\varepsilon$ we have:
\[ \lambda(\nL_{\Z_2}; s) = \prod_{i=1}^{\lfloor\frac{m-1}{2}\rfloor}(1-2^{-2i-2s})(1-(\frac{\varepsilon}{2})2^{-\frac{m}{2}-s}). \]
Here $(\frac{\varepsilon}{2}) = (-1)^{\frac{\varepsilon^2-1}{8}}$ is the Kronecker symbol.

\item Let $s \in \Z_{\ge 0}$. For a lattice $\nL_{\Z_2}$ of the form $\nL' \perp <\varepsilon'>$, where $\nL'_{\Z_2}$ is unimodular of discriminant $\varepsilon$, 
and $\varepsilon, \varepsilon' \in \Z_2^*$, we have:
\[ \lambda(\nL_{\Z_2}; s) = |2|_2^{s+\frac{m-1}{2}} \prod_{i=1}^{\lfloor\frac{m-1}{2}\rfloor}(1-2^{-2i-2s}). \]

\item Let $\nL_{\Z_2}$ be a lattice of the form $<\varepsilon_1,\varepsilon_2>, \varepsilon_i \in \Z_2^*$. We have
\[ \lambda(\nL_{\Z_2}; 0) = \frac{1}{2}. \]
\end{enumerate}
\end{SATZ}

\section{A simple illustration of the orbit equation}\label{SIMPLEILLUSTRATION}

\begin{PAR}
We give an easy example to illustrate the orbit equation:

Let $p\not=2$ and $\nL_\Zp$, be a unimodular lattice of {\em odd} dimension $m \ge 3$. We want to calculate
\[ \mu_p(\nL_\Zp, <\varepsilon p^a>, \nL_\Zp; s) = |p^a|_p^{-s+\frac{2-m}{2}} \beta_p(\nL_\Zp, <\varepsilon p^a>, \nL_\Zp; s). \]
We first assume, $a$ {\em odd}. One the one hand, by Yang's formula (Theorem \ref{YANG}), it is given by 
\[
|p^a|_p^{-s+\frac{2-m}{2}} \left(1+(1-p^{-1})\sum_{k=1}^{\frac{a-1}{2}}p^{(2-m)k}X^{2k} - p^{(2-m)\frac{a+1}{2}-1}X^{a+1}\right),
\]
where $X=p^{-s}$, as usual.
On the other hand, we have $\frac{a+1}{2}$ orbits of vectors of length $\varepsilon p^a$ (Lemma \ref{LEMMA1ORBIT}). The lattices $\alpha_i(<\varepsilon p^a>)^\perp$ are 
of the form \mbox{$<\epsilon'' p^{2i-1}>$}$\perp \nL_\Zp'$, where $\nL_\Zp'$ is unimodular. We have by Theorem \ref{EXPLIZIT}:
\[ \lambda(\alpha_i(<\varepsilon p^a>)^\perp; s) = |p^{2i-1}|^{s+\frac{m-2}{2}} \prod_{j=1}^{\frac{m-3}{2}} (1-p^{-2j-2s}) \]
and
\[ \lambda(\nL_\Zp; s) = \prod_{i=1}^{\frac{m-1}{2}}(1-p^{-2i-2s}). \]
The orbit equation hence reduces to the following identity of rational functions in $X=p^{-s}$,
\begin{eqnarray*} 
&& \prod_{i=1}^{\frac{m-1}{2}}(1-p^{-2i}X^{2})^{-1}  \\
&\cdot& X^{-a} p^{a\frac{m-2}{2}} \left(1+(1-p^{-1})\sum_{k=1}^{\frac{a-1}{2}}p^{(2-m)k}X^{2k} - p^{(2-m)\frac{a+1}{2}-1}X^{a+1}\right)  \\
&=& \sum_{i=1}^{\frac{a+1}{2}} X^{-(2i-1)} p^{(2i-1)\frac{m-2}{2}} \prod_{j=1}^{\frac{m-3}{2}} (1-p^{-2j}X^2)^{-1},
\end{eqnarray*}
which one can check easily in an elementary way.

If $a$ is even, let $2\varepsilon'$ be the discriminant of $\nL_\Zp$. By Yang's formula (Theorem \ref{YANG}), we get:
\[ |p^a|_p^{-s+\frac{2-m}{2}} \left(1 + (1-p^{-1})\sum_{k=1}^\frac{a}{2}p^{(2-m)k}X^{2k} + \left( (-1)^{\frac{m-1}{2}}\frac{\varepsilon \varepsilon'}{p} \right) p^{\frac{(2-m)(a+1)-1}{2}}X^{a+1}\right). \]
On the other hand, we have $\frac{a}{2}+1$ orbits of vectors of length $\varepsilon p^a$ (Lemma \ref{LEMMA1ORBIT}). The lattices $\alpha_i(<\varepsilon p^a>)^\perp$ are 
of the form $<\varepsilon'' p^{2i}> \perp \nL_\Zp'$, where $\nL_\Zp'$ is unimodular. For $i=0$, $\alpha_i(<\varepsilon p^a>)^\perp$ has discriminant
$\varepsilon\varepsilon'$.
We have by Theorem \ref{EXPLIZIT}:
\[ \lambda(\alpha_i(<\varepsilon p^a>)^\perp; s) = \prod_{j=1}^{\frac{m-3}{2}} (1-p^{-2j-2s}) \begin{cases} 
1-\left(\frac{(-1)^{\frac{m-1}{2}} \varepsilon\varepsilon'}{p} \right) p^{-\frac{m-1}{2}-s} & i=0, \\ 
 |p^{2i}|^{s+\frac{m-2}{2}} & i>0. \end{cases} \]
The orbit equation hence reduces to the following identity of rational functions in $X=p^{-s}$,
\begin{eqnarray*} 
&& \prod_{i=1}^{\frac{m-1}{2}}(1-p^{-2i}X^{2})^{-1}  \\
&\cdot& X^{-a} p^{a\frac{m-2}{2}} \left( 1 + (1-p^{-1})\sum_{k=1}^\frac{a}{2}p^{(2-m)k}X^{2k} + \left( \frac{(-1)^{\frac{m-1}{2}} \varepsilon' \varepsilon}{p} \right) p^{\frac{(2-m)(a+1)-1}{2}}X^{a+1} \right)  \\
&=& \left( \frac{1}{1-\left(\frac{(-1)^{\frac{m-1}{2}}\varepsilon\varepsilon'}{p}\right)p^{-\frac{m-1}{2}}X} + \sum_{k=1}^{\frac{a}{2}} X^{-2k} p^{k(m-2)} \right) \prod_{j=1}^{\frac{m-3}{2}} (1-p^{-2j}X^2)^{-1},
\end{eqnarray*}
which one can check easily in an elementary way.
 
\end{PAR}

\section{A global orbit equation}\label{SECTGLOBALORBIT}

\begin{LEMMA}
Let $\nM_\Q$, $\nL_\Q$ be vector spaces of dimensions $n$, $m$ respectively, where $n \le m-1$, with quadratic forms $Q_\nM$ and $Q_\nL$.
If $m \ge n+3$, the product of the canonical measures on $I(\nM_{\Q_\nu}, \nL_{\Q_\nu})$ converges absolutely (in the sense of \cite{Weil3}) and yields
the canonical measure on $I(\nM_{\A}, \nL_{\A})$.

In the case $n=m$, the product of the canonical measures on $\SO(\nL_{\Q_\nu})$ converges absolutely and yields the 
the canonical measure on $\SO(\nL_{\A})$, provided $m \ge 3$. It is the Tamagawa measure.
\end{LEMMA}
\begin{proof}
Follows directly from the explicit volume formul\ae\ in the unimodular case (Theorem~\ref{EXPLIZIT}, 4., 5.) and standard facts about absolute convergence of
the occurring infinite products. One just obtains the Tamagawa number in the second case because of the product formula $|x|_\A=1$ 
for $x\in \Q^*$ for the adelic modulus, i.e. the discriminant factors cancel in the product.
\end{proof}

\begin{PAR}\label{GLOBALEISENSTEIN}
Let $\nL_\Z$ be a lattice of dimension $m$ with $Q_\nL \in \Sym^2(\nL_\Z^*)$ of signature $(m-2,2)$.

For 1-dimensional $\nM \cong \Z$, $(n=1)$, we have $Q \in \Z$ describing the quadratic form $x \mapsto Qx^2$ with associated `symmetric morphism' $\gamma = 2Q$, and
$\gamma_0=1$:
\begin{gather*}
\lim_{\alpha \rightarrow \infty} |\alpha|^{-\frac{m}{2}} e^{2\pi \alpha^2 Q} E_{Q}(\Psi_{\infty, \frac{m}{2}} \Phi(\chi_\kappa);  g_l(\alpha), s)  \\
= \lim_{\alpha \rightarrow \infty} |\alpha|^{-\frac{m}{2}} e^{2\pi \alpha^2 Q}  W_{Q,\infty}(\Psi_{\infty, \frac{m}{2}}(s), g_l(\alpha)) 
\prod_p W_{Q, p} (\Phi_p(\chi_\kappa; s), 1)  \nonumber \\
= |\gamma|_\infty^{s} |2\gamma|_\infty^{\frac{1}{2}(s-s_0)}  \mu_\infty(\nL_\Zp, <Q>, \kappa; s-s_0) \prod_p |\gamma|_p^{s} \mu_p(\nL_\Zp, <Q>, \kappa; s-s_0).
\end{gather*}
(cf. \ref{REPDENSNONARCHIMEDIANINTERPOLATION} and \ref{SATZREPDENSARCHIMEDIANINTERPOLATION}, respectively).

This is the quantity, which can be related to Arakelov geometry, using the result of \cite{Kudla4} or \cite{BK} (see section \ref{KUDLA}).

We therefore define
\[ \mu(\nL_\Z, \nM_\Z, \kappa; s) = \prod_\nu \mu_\nu(\nL_\Z, \nM_\Z, \kappa; s) \]
and
\[ \widetilde{\mu}(\nL_\Z, \nM_\Z, \kappa; s) = |2 d(\nM_\Z)|_\infty^{-\frac{1}{2}s} \mu(\nL_\Z, \nM_\Z, \kappa; s). \]
With this definition we have in the 1-dimensional case:
\begin{eqnarray*}\label{es}
\lim_{\alpha \rightarrow \infty} |\alpha|^{-\frac{m}{2}} e^{2\pi \alpha^2 Q} E_{Q}(\Psi_\infty \Phi(\chi_\kappa);  g_l(\alpha), s-s_0) &=& \widetilde{\mu}(\nL_\Zp, <Q>, \kappa; s).
\end{eqnarray*}
\end{PAR}
For $n>1$ we do not now, if the global $\widetilde{\mu}$ occurs as a limit in this fashion as well.

\begin{PAR}\label{PARGLOBALORBITEQUATION}
Let $D$ be the discriminant of $\nL$. Assume, that $\nL_\Zp$ splits $n$ hyperbolic planes at all $p\not=2$. Define
\[ \lambda(\nL_\Z; s) := \prod_\nu \lambda_\nu(\nL_\Z; s) \]
and
\[ \widetilde{\lambda}(\nL; s) := |D|_\infty^{\frac{1}{2}s} \lambda(\nL_\Z; s), \] 
where $D$ is the discriminant of $\nL_\Z$. Sometimes we will use also $\widetilde{\lambda}_p(\dots)$ for $|D|_p^{-\frac{1}{2}s} \lambda_p(\dots)$, 
and similarly $\widetilde{\mu}_p(\dots):=|2 d(\nM_\Z)|_p^{\frac{1}{2}s}\mu_p(\dots)$.
\end{PAR}

\begin{LEMMA}If $\nL_\Z$ is of signature $(m-2,2)$. Let $\nM_\Z$ be a positive definite lattice. Take a $\kappa \in (\nL_\Zh^*/\nL_\Zh) \otimes \nM_\Zh^*$.
Assume $m-n\ge 1$ and $\nIsome(\nM, \nL)(\Af) \cap \kappa \not= \emptyset$.

$\mu(\nL_\Z, \nM_\Z, \kappa; s),  \lambda(\nL_\Z; s)$ have meromorphic continuations to the entire complex plane and are holomorphic and nonzero in a neighborhood of $s=0$. Similarly for $\widetilde{\mu}, \widetilde{\lambda}$. They depend only on the genera of $\nL_\Z, \nM_\Z$.
\end{LEMMA}
\begin{proof}
Follows directly from (\ref{EXPLIZIT}) and using standard facts about the occurring quadratic $L$-series.
\end{proof}
In particular, for the meromorphic continuation of the Eisenstein series (\ref{es}) above remains true also in exceptional cases.

We obviously get taking the product over all $\nu$ of equations (\ref{BAHNENGLEICHUNG}) and (\ref{BAHNENGLEICHUNGARCHIMEDISCH}), respectively: 
\[ \lambda^{-1}(\nL_\Z; s) \mu(\nL_\Z, \nM_\Z, \kappa; s) = \sum_{\alpha \SO'(\nL_\Zh) \subset \nIsome(\nM, \nL)(\Af) \cap \kappa} \lambda^{-1}(\alpha^\perp_\Z; s), \]
where $\alpha^\perp_\Z$ is a lattice such that $\alpha^\perp_\Z \otimes \Zp \cong \alpha^\perp_\Zp$ for all $p$. It can be realized in some class $\nL_\Z'$ in the genus of $\nL_\Z$
(\ref{DEFI}). $\lambda$ depends only on its genus.

For the corresponding $\widetilde{\mu}, \widetilde{\lambda}$, this equation is not true anymore. However, we get at least, denoting by $\R_N$ the reals modulo
rational multiples of $\log(p)$, for $p \mid N$:
\begin{SATZ}\label{GLOBALORBITEQUATION}
Assume $m \ge 3$, $m-n \ge 1$. Let $D$ be the discriminant of $\nL_\Z$ and $D'$ be the $D$-primary part of the discriminant of $\nM_\Z$.
\[ \widetilde{\lambda}^{-1}(\nL_\Z; 0) \widetilde{\mu}(\nL_\Z, \nM_\Z, \kappa; 0) = \sum_{\alpha \SO'(\nL_\Zh) \subset \nIsome(\nM, \nL)(\Af) \cap \kappa} \widetilde{\lambda}^{-1}(\alpha^\perp_\Z; 0) \]
and
\[ \frac{\dd}{\dd s} \left. { \left(\widetilde{\lambda}^{-1}(\nL_\Z; s) \widetilde{\mu}(\nL_\Z, \nM_\Z, \kappa; s) \right) } \right|_{s=0} \equiv \sum_{\alpha\SO'(\nL_\Zh) \subset \nIsome(\nM, \nL)(\Af) \cap \kappa} \frac{\dd}{\dd s} \left.  \widetilde{\lambda}^{-1}(\alpha^\perp_\Z; s) \right|_{s=0} \]
in $\R_{2DD''}$, where $D''$ is the product of primes such that $p^2 \nmid D'$.
\end{SATZ}
\begin{proof}
Follows by induction from the fact that for $p \nmid D$ and $Q$ square-free at $p$ there is only 1 orbit (\ref{LEMMA1ORBIT}) in $\nIsome(<Q>, \nL)(\Zp)$ and $\alpha^\perp$ has discriminant $p$ at $p$.
\end{proof}

We will need also a global version of Kitaoka's formula (\ref{KITAOKA}):

\begin{SATZ}\label{GLOBALKITAOKA}
Assume $m \ge 3$, $m-n \ge 1$. Let $D$ be the discriminant of $\nL_\Z$ and $\nM_\Z = \nM_\Z' \perp \nM_\Z''$. Let $D'$ be the $D$-primary part of the 
discriminant of $\nM_\Z'$ (not $\nM_\Z$ !).
Let $\kappa \in (\nL_\Z^*/\nL_\Z) \otimes \nM_\Z^*$ with a corresponding decomposition $\kappa = \kappa' \oplus \kappa''$.
We have
\[ \widetilde{\lambda}^{-1}(\nL_\Z; 0) \widetilde{\mu}(\nL_\Z, \nM_\Z, \kappa; 0) = \sum_{\substack{\alpha \SO'(\nL_\Zh) \subset \nIsome(\nM', \nL)(\Af) \cap \kappa' \\ \kappa'' \cap \alpha^\perp_\Af \otimes (\nM''_\Af)^* \not= \emptyset}} \widetilde{\lambda}^{-1}(\alpha^\perp_\Z; 0) \widetilde{\mu}(\alpha^\perp_\Z, \nM_\Zh, \kappa''; 0) \]
and
\begin{gather*} 
\frac{\dd}{\dd s} \left. { \left( \widetilde{\lambda}^{-1}(\nL_\Z; s) \widetilde{\mu}(\nL_\Z, \nM_\Z, \kappa; s) \right) } \right|_{s=0} \\ = 
\sum_{\substack{\alpha \SO'(\nL_\Zh) \subset \nIsome(\nM', \nL)(\Af) \cap \kappa' \\ \kappa'' \cap \alpha^\perp_\Af \otimes (\nM''_\Af)^* \not= \emptyset}} \frac{\dd}{\dd s} \left. { \left( \widetilde{\lambda}^{-1}(\alpha^\perp_\Z; s) \widetilde{\mu}(\alpha^\perp_\Z, \nM'', \kappa''; s)
\right) } \right|_{s=0} 
\end{gather*}
in $\R_{2DD'}$. Here $\kappa'' \in \nL_\Zh^*/\nL_\Zh$ is considered as an element of $(\alpha^\perp_\Zh)^* / \alpha^\perp_\Zh \otimes (\nM''_\Zh)^*$ via 
$\kappa'' \mapsto \kappa'' \cap \alpha^\perp_\Af \otimes (\nM''_\Af)^*$.
\end{SATZ}
\begin{proof}
Let first $s \in \Z_{\ge 0}$.
For all $p \nmid 2DD'$, there is only one orbit (generated for $s=0$ by $\alpha$, say) in $\nIsome(\nM', \nL\oplus H^s)(\Zp)$. Hence we have, 
using Kitaoka's formula \ref{KITAOKAINTERPOLATED}:
\[ \lambda^{-1}_p(\nL_\Zp; s) \widetilde{\mu}_p(\nL_\Zp, \nM_\Zp, \nL_\Zp; s) = \lambda^{-1}_p(\alpha^\perp_\Zp; s) \widetilde{\mu}_p(\alpha^\perp_\Zp, \nM_\Zp'', \kappa''; s). \]
For all other $p$, we have the equation 
\[ \lambda^{-1}_p(\nL_\Zp; 0) \mu_p(\nL_\Zp, \nM_\Zp, \kappa; 0) = 
\sum_{\substack{\alpha \SO'(\nL_\Zp) \subset \nIsome(\nM', \nL)(\Qp) \cap \kappa' \\ \kappa'' \cap \alpha^\perp_\Qp \otimes (\nM''_\Qp)^* \not= \emptyset}} \lambda^{-1}_p(\alpha^\perp_\Zp; 0) \mu_p(\alpha^\perp_\Zp, \nM'', \kappa''; 0). \]
and since the quantities $\widetilde{\mu}_p$ are polynomials in $p^{-s}$ in this case (cf. \ref{REPDENSNONARCHIMEDIANINTERPOLATION}),
the assertion is true.
\end{proof}

\section{Applications to Kudla's program}\label{KUDLASPROGRAM}

\begin{PAR}\label{INTEGRALMODELS}
The investigations of recursive properties of representation densities in this paper were 
motivated by the program of Kudla (cf. \cite{Kudla1, Kudla2, Kudla3, Kudla5, Kudla6, KR1, KR2, KR3, KRY1, KRY2, KRY3}) to describe the relation of heights of special subvarieties of
certain Shimura varieties to special derivatives of the Eisenstein series considered in section \ref{DEFEISENSTEIN} of this paper. 
We briefly describe this here and report on related results obtained in the thesis of the author \cite{Thesis}, which will be published in a forthcoming paper \cite{Paper2}.
The object of study are the orthogonal Shimura varieties, which are defined over $\Q$ ($m \ge 3$) and whose associated complex analytic orbifolds are given by
\[ [ \SO(\nL_\Q) \backslash \nX_\nO \times (\SO(\nL_\Af) / K) ], \]
where $\nL_\Q$ is a quadratic space of signature $(m-2,2)$, $\nX_\nO$ is the associated (Hermitian) symmetric space and $K$ is a compact open subgroup of
$\SO(\nL_\Af)$. There is an injective intertwining map 
$\nh_\nO: \nX_\nO \hookrightarrow \Hom(\SSS, \SO(\nL_\R))$ such that $\nO:=(\nP_\nO:=\SO(\nL), \nX_\nO, \nh_\nO)$ is a Shimura datum \cite[Definition 3.2.2]{Thesis}.
If a lattice $\nL_\Z$ is chosen, $\nO$ is integral at all primes not dividing the discriminant $D$ of $\nL_\Z$ (this means, that $\nP_\nO$ extends to a reductive group scheme over $\Zpp$).  In the first part of the thesis of the author \cite[Main Theorems 4.2.2, 4.3.5]{Thesis}, it is sketched that there is a {\em canonical} integral model $\nSh({}^K_\nRPCD \nO)$ of a toroidal compactification of the above variety defined over $\Z[1/2D]$  ($m \ge 3$) provided $K$ is admissible, i.e. for all $p \nmid D$ of the form $\SO(\nL_\Zp) \times K^{(p)}$, where $K^{(p)}$ is a compact open subgroup of $\SO(\nL_\Afp)$\footnote{The proofs still rely on a technical assumption which remains open.}.  The compactification depends on the additional datum of a rational polyhedral cone decomposition $\nRPCD$. In \cite[Main Theorem 4.5.2]{Thesis} it is also shown, that there exists a theory of 
Hermitian automorphic vector bundles on these models. For each $\SO$-equivariant vector bundle $\mathcal{E}$ on 
\[ \nShD(\nO) = \{ <v> \in \PP \nL \where Q(v)=0 \} \] 
(the compact dual, defined over $\Z[1/2D]$, too) equipped with a $\SO(\nL_\R)$-invariant Hermitian metric on $\mathcal{E}_\C$, restricted to the image of the Borel embedding, there is an associated  (canonically determined) Hermitian automorphic vector bundle $(\Xi^*\mathcal{E}, \Xi^*h)$  on $\nSh({}^K_\nRPCD \nO)$ whose metric has logarithmic singularities along the
boundary divisor of $\nSh({}^K_\nRPCD \nO)(\C)$. The associated analytic bundle, restricted to the uncompactified Shimura variety, can be canonically identified with 
\[ [\SO(\nL_\Q) \backslash \mathcal{E}_\C|_{\nX_\nO} \times (\SO(\Af) / K)], \]
equipped with the quotient of the given metric. 
The construction is functorial in morphisms of Shimura data, in particular, everything above commutes with inclusions of lattices.
The results so far are conditional on a missing technical hypothesis \cite[Conjecture 4.3.2]{Thesis}.
In what follows, we let $\mathcal{E}$ be the restriction of the tautological bundle on $\PP \nL$ and $h$ be a certain multiple \cite[Definiton 11.3.1]{Thesis} of the 
metric $v \mapsto \langle v, \overline{v} \rangle$. 
\end{PAR}

\begin{PAR}\label{SPECIALCYCLE}
Let $\nM$ be a positive definite space and $\varphi \in S(\nL_\Af \otimes \nM_\Af^*)$ a locally constant function with compact support.
Provided that $\varphi$ is $K$-invariant, this defines a {\bf special cycle} $\nZ(\nL,\nM,\varphi; K)$ on these Shimura varieties, defined as
\begin{gather*}
 \sum_{ K \alpha \subset \nIsome(\nM,\nL)(\Af) \cap \supp(\varphi) } \varphi(\alpha) \left[ \SO(\alpha^\perp_\Q) \backslash \nX_{\nO(\alpha^\perp)} \times \SO(\alpha^\perp_\Af) / ({}^gK \cap \SO(\alpha^\perp_\Af)) \right].
\end{gather*}
Observe that the sum goes over finitely many orbits and that there is a natural embedding $\nX_{\nO(\alpha_\Z^\perp)} \hookrightarrow \nX_\nO$ and
an embedding $\SO(\alpha^\perp_\Af) \rightarrow \SO(\nL_\Af), h \mapsto hg^{-1}$, where $g$ is the element defining $\alpha^\perp_\Z$ as an integral lattice (\ref{DEFI}). 
These cycles are naturally a weighted sum over sub-Shimura varieties of the same type. In the special case $K = \SO'(\nL_\Zh)$ and $\varphi$ is the characteristic function of a
coset in $(\nL_\Zh^*/\nL_\Zh) \otimes \nM_\Zh^*$, these correspond precisely to the orbits in the global orbit equation \ref{GLOBALORBITEQUATION}.
\end{PAR}

\begin{PAR}
We have the following relation between {\em values} at 0 of the orbit equation and volumes of special cycles. The equation says
\[ 4\lambda^{-1}(\nL_\Z; 0) \mu(\nL_\Z, \nM_\Z, \kappa; 0) = 4\sum_{\SO'(\nL_\Zh)\alpha \subset \nIsome(\nM, \nL)(\Af) \cap \kappa} \lambda^{-1}(\alpha^\perp_\Z; 0). \]
It is well known that $4\lambda^{-1}(\nL_\Z, 0)$ equals the volume of  $\nSh({}^{\SO'(\nL_\Zh)} \nO)(\C)$ w.r.t. the volume form  $\chern_1(\Xi^*\mathcal{E}, \Xi^* h)^{m-2}$ --- 
see e.g. \cite[Theorem 11.5.2, i.]{Thesis}. Similarly $4\lambda^{-1}(\alpha^\perp_\Z; 0)$ is the volume of  the sub-Shimura variety of the
special cycle.
Hence by definition of the special cycle (and canonicity of the construction $\Xi^*$), $\mu(\nL_\Z, \nM_\Z, \kappa; 0)$ 
is equal to the degree of the special cycle (w.r.t. $\Xi^*\mathcal{E}, \Xi^* h$) divided by the volume of the surrounding 
Shimura variety. More generally, it is possible to define the special cycles also for degenerate
quadratic forms on $\nM$ and the equality with the special values of the Fourier coefficients of the Eisenstein series extends to those. 
Therefore the special value of the Eisenstein series is a generating function for these degrees.  
In a lot of cases, it is known that in fact also their generating functions valued in cohomology or even Chow groups of the Shimura variety 
are modular. See the introduction to \cite{Thesis} and the references therein for more information on this. 
\end{PAR}

\begin{PAR}
According to the conjectures of Kudla and others, the first derivative of $\widetilde{\mu}$, or more generally of the full Fourier coefficient of the Eisenstein series \ref{DEFEISENSTEIN}, should be related to {\em heights} of the special cycles. In \cite[Theorem 11.5.2, ii., 11.5.5, 11.5.9]{Thesis}, we proved a partial result for the nonsingular coefficients in this direction:
\end{PAR}
\begin{SATZ}\label{HEIGHTSPECIAL}
Assume $m-n>1$.
Under a general technical assumption (cf. \ref{INTEGRALMODELS}), we have the following:
\begin{center}
\small
\begin{tabular}{c|c|c}
& $\nSh({}^{\SO'(\nL_\Zh)}\nO)$ & $\nZ(\nL_\Z, \nM_\Z, \kappa)$ \\
\hline
&&\\
\bf geometric& & \\
 \bf volume& $\prod_\nu  \widetilde{\lambda}_\nu^{-1}(\nL; 0)$ & $\prod_\nu  \widetilde{\lambda}_\nu^{-1}(\nL; 0)   \widetilde{\mu}_\nu(\nL, \nM, \kappa; 0)$ \\
 w.r.t. && \\
$\chern_1(\Xi^*\mathcal{E}_\C, \Xi^* h)$ && \\
\hline
&& \\
\bf arithmetic & & \\
\bf volume & $ \left. \frac{\dd}{\dd s} \prod_\nu  \widetilde{\lambda}_\nu^{-1}(\nL; s) \right|_{s=0}$  & $ \left. \frac{\dd}{\dd s} \prod_\nu \widetilde{\lambda}_\nu(\nL; s)   \widetilde{\mu}_\nu(\nL, \nM, \kappa; s) 
\right|_{s=0}$ \\
w.r.t & & \\
 $\achern_1(\Xi^*\mathcal{E}, \Xi^* h)$& {\footnotesize up to $\Q \log(p)$} & {\footnotesize up to $\Q \log(p)$ for $p | 2d(\nL)$ and } \\ 
& {\footnotesize for $p^2 | 4d(\nL)$.} & {\footnotesize $p$ such that $\nM_\Zp^*/\nM_\Zp$ is not cyclic.} \\ 
\end{tabular}
\end{center}
\end{SATZ}
Here, in the first line, we repeated the well-known facts for the case of the geometric volume, mentioned above, for the sake of comparison.

Note that $\widetilde{\mu}(\nL_\Z, \nM_\Z, \kappa; s)$ is the ``holomorphic'' part of the Fourier coefficient of the Eisenstein series \ref{DEFEISENSTEIN}. 
It is conjectured, that if one equips the special cycles with Kudla-Millson's Greens functions \cite{KuMi1, KuMi2} and extends them (possibly in a nontrivial way, i.e. not by simply taking its Zariski closure) to the integral model, the generating series of their classes $[\widehat{\nZ}(\nL_\Z, \nM_\Z, \kappa, y)]$ in appropriate arithmetic Chow groups should be modular forms, whose product with $\achern_1(\Xi^*\mathcal{E}, \Xi^* h)^{m-n-1}$ should yield the {\em full} first derivative of the Eisenstein series. 
The Greens functions depend on $y \in (\nM \otimes \nM)_\R^s$ (the imaginary part of the argument of the Eisenstein series as classical modular form), too.

Whereas $\achern_1(\Xi^*\mathcal{E}, \Xi^* h)$ is an object in 
an already developed \cite{BKK1, BKK2} extended Arakelov theory allowing additional boundary singularities of the occurring Greens function, this is not yet clear for these classes $[\widehat{\nZ}(\nL_\Z, \nM_\Z, \kappa, y)]$. In any case, the additional ingredient would be to prove that the integral of the Kudla-Millson Greens function is given 
by the special derivative of the ``nonholomorphic part'' of the Fourier coefficient of the Eisenstein series in the following way:

\begin{gather*}  
\int_{\nSh({}^K\nO)(\C)} \mathfrak{g}(Q, y) \chern_1(\Xi^* \mathcal{E}, \Xi^* h)^{m-n-1}  \\
= \vol(\nZ(\nL_\Z, \nM_\Z^Q, \kappa; K)) \frac{\dd}{\dd s} \left. \mu_{nh}(\nL_\Z, \nM_\Z^Q, \kappa; y, s) \right|_{s=0},
\end{gather*}
where $\mu_{nh}(\nL_\Z, \nM_\Z^Q, \kappa; y, s)$, the ``non-holomorphic'' part of the Fourier coefficient of the Eisenstein series is determined by
\begin{gather*} 
|\alpha|^{-\frac{m}{2}} e^{2\pi \alpha^! Q \cdot \gamma_0^{-1}} E_{Q}(\Psi_\infty \Phi(\chi_\kappa);  g_l(\alpha), s-s_0) \\
= \widetilde{\mu}(\nL_\Z, \nM_\Z^Q, \kappa; s) \mu_{nh}(\nL_\Z, \nM_\Z^Q, \kappa;  \alpha \gamma_0^{-1} ({}^t \alpha^{-1}), s).
\end{gather*}
Its value at $s=0$ is equal to 1, independently of $\alpha$ (\ref{GLOBALEISENSTEIN}). See e.g. \cite[Proposition 12.1]{KRY2} for a calculation of this kind for Shimura curves, which correspond to signature (1,2), Witt rank 0.

At primes, where the Zariski closure of the special cycle itself consists of canonical integral models of Shimura varieties, the height is equal (again by the compatibility of the $\Xi^*$ construction with embeddings) to the (degree of the) highest power of $\achern_1(\Xi^*\mathcal{E}, \Xi^* h)$ computed on these various smaller models. 

By Theorem \ref{HEIGHTSPECIAL}, this ``self-intersection number'' or ``arithmetic volume'' of a Shimura variety associated with a lattice $\nL'_\Z$ and its discriminant kernel
is given by the first derivative at $s=0$ of 
\[ 4\widetilde{\lambda}^{-1}(\nL_\Z'; s)   \]
up to multiples of $\log(p)$ dividing the discriminant of $\nL'_\Z$ and $p=2$. For induced (but maybe not canonical in any reasonable sense) models one can extend this
to be up to multiples of those $\log(p)$, where $p^2$ divides the discriminant.
Obviously, the orbit equation relates this --- applied to the individual sub-Shimura varieties in a special cycle --- directly to the height (or geometric volume) of the special cycle.
We emphasize, that the contribution of the RHS of the orbit equation {\em for very bad primes} is {\em not} conjectured to be related to the decomposition into sub-Shimura varieties.
This is already wrong in the case of the modular curve (Shimura variety associated with the lattice considered in section \ref{MODULARCURVE}) as we will illustrate in \ref{SCHEITERN} below.

\begin{PAR}\label{KUDLA}
For $n=1$, the global orbit equation at $s=0$ and its derivative can be understood in terms of Borcherds products (and in the arithmetic case this is the key to the proof of \ref{HEIGHTSPECIAL}!). This idea was
first used in \cite{BBK}.
Recall that a Borcherds product is a meromorphic modular form on $\nSh({}^K \nO)(\C)$ having singularities precisely in the (codimension $n=1$) special cycles. It is a multiplicative lift of an integral vector valued modular form $f$ of weight $1-\frac{m}{2}$ holomorphic in $\HH$ and meromorphic in the cusp $i \infty$ for the Weil representation 
of $\Sp_2'(\Z)$ restricted to $\C[\nL_\Zh^*/\nL_\Zh]$. Such an $f$ has a Fourier expansion
\[ f = \sum_{k \in \Q_{\gg -\infty}} a_k q^k, \]
where only $a_k$`s with $k$ having bounded denominator actually occur and $a_k \in \Z[\nL_\Zh^*/\nL_\Zh].$ The divisor of $F$ (on the uncompactified Shimura variety) is given by $\sum_{k<0} \nZ(\nL_\Z, <-k>, a_k)$ and its weight is equal to $a_0(0)$.

The geometric case is now ``explained'' as follows: Using Serre duality one can show, that the nonpositive Fourier coefficients have to satisfy the equation
\[ \sum_{k \in \Q_{\le 0}} a_k b_{-k} = 0 \]
for all {\em holomorphic} forms $\sum_{k \in \Q_{\ge 0}} b_k q^k$ of weight $\frac{m}{2}$ for the dual of the Weil representation restricted to $\C[\nL_\Zh^*/\nL_\Zh]$.
The Eisenstein series (see \ref{GLOBALEISENSTEIN}), whose special value is holomorphic (assuming $m\ge 4$), yields the relation
\[ c_0(0) = \sum_{k<0} \mu(\nL_\Z, <-k>, a_k; 0). \]
If we assume for the moment that there is only one nonzero $a_k, k<0$ equal to the characteristic function of $\kappa$, we may confirm that 
$\mu(\nL_\Z, <-k>, \kappa; 0)$ gives the relative degree of $\nZ(\nL_\Z, <-k>, \kappa)$ w.r.t. $\Xi^*\mathcal{E}, \Xi^* h$ (modulo investigations at the boundary).

The arithmetic case is ``explained'' as follows: We assume again for simplicity that there is only one nonzero $a_k, k<0$ equal to the characteristic function of a $\kappa$.
In \cite[Theorem 11.3.11]{Thesis} we showed (using Borcherds' product expansion and an abstract integral $q$-expansion principle for automorphic vector bundles) that for good $p$ these Borcherds forms are actually rational sections of the bundle $\Xi^* \mathcal{E}$ on any smooth compactified integral model $\nSh({}^K_\nRPCD \nO)$, having no connected (=irreducible) fibre above $p$ in their divisor. Their divisor hence is supported on the Zariski closures of the special divisiors and possibly on the boundary. 
The derivative of the orbit equation now reads as follows:
\begin{gather*}
4(\widetilde{\lambda}^{-1})'(\nL_\Z; 0) \widetilde{\mu}(\nL_\Z, <-k>, \kappa; 0) + 4\widetilde{\lambda}^{-1}(\nL_\Z; 0) \widetilde{\mu}'(\nL_\Z, <-k>, \kappa; 0) \\
= 4\sum_{\SO'(\nL_\Zh) \alpha \subset \nIsome(<Q>, \nL)(\Af) \cap \kappa } (\widetilde{\lambda}^{-1})'(\alpha^\perp_\Z; 0).
\end{gather*}
$4(\widetilde{\lambda}^{-1})'(\alpha^\perp_\Z; 0)$
is interpreted as the arithmetic volume/height of the sub-Shimura variety in the special divisor corresponding to the orbit of $\alpha$. Hence the value of the equation is the height of the special divisor (cf. Theorem \ref{HEIGHTSPECIAL}). 
\[ 4(\widetilde{\lambda}^{-1})'(\nL_\Z; 0) \widetilde{\mu}(\nL_\Z, <-k>, \kappa; 0) \]
is the arithmetic volume of the surrounding Shimura variety multiplied by the weight of $F$, and finally, according to the computation of Kudla and Bruinier/K\"uhn (compare \ref{GLOBALEISENSTEIN}
with \cite[Theorem 2.12 (ii)]{Kudla4} or use \cite[Proposition 3.2, Theorem 4.11]{BK} and \ref{REPDENS})
\[ -4\widetilde{\lambda}^{-1}(\nL_\Z; 0) \widetilde{\mu}'(\nL_\Z, <-k>, \kappa; 0) \] 
is equal to the integral of the logarithm of the Hermitian norm of $F$ over $\nSh({}^K \nO)(\C)$ (taking into account, that $4\widetilde{\lambda}^{-1}(\nL_\Z; 0)$ is the volume of $\nSh({}^K \nO)(\C)$). The equation hence (again modulo --- considerably difficult --- investigations at the boundary) expresses the following standard relation in Arakelov geometry:
\begin{eqnarray*} 
\adeg(\widehat{\Div}(F) \cdot \achern_1(\Xi^* \mathcal{E}, \Xi^* h)^{m-2}) &=& \adeg(\achern_1((\Xi^* \mathcal{E}, \Xi^* h)|_{\overline{\nZ(\nL_\Z, <-k>, \kappa)}})^{m-2}) \\
&+& \int_{\nSh({}^K \nO)(\C)} \log h(F) \chern_1(\Xi^* \mathcal{E}, \Xi^* h)^{m-2}.
\end{eqnarray*}
Observe $\widehat{\Div}(F) = \mu(\nL_\Z, <-k>, \kappa; 0) \achern_1(\Xi^* \mathcal{E}, \Xi^* h)$, so the LHS is equal to $\mu(\nL_\Z, <-k>, \kappa; 0) \adeg( \achern_1(\Xi^* \mathcal{E}, \Xi^* h)^{m-1} )$.
All values are understood in $\R$ modulo contribution from $\log(p)$ for $p$ dividing the discriminant of $\nL_\Z$ (not $\nM_\Z=<-k>$!) and for $p=2$.
\end{PAR}

\begin{PAR}
Consider a decomposition $\nM_\Z = \nM_\Z' \perp \nM_\Z''$.
The derivative of the global Kitaoka formula \ref{GLOBALKITAOKA} at $s=0$, which reads 
\[  4\widetilde{\lambda}^{-1}(\nL_\Z; s) \widetilde{\mu}(\nL_\Z, \nM_\Z, \kappa; s) = 4
\sum_{\substack{\SO'(\nL_\Zh)\alpha \subset \nIsome(\nM', \nL)(\Af) \cap \kappa' \\ \kappa'' \cap \alpha_\Af^\perp \otimes (\nM_\Af'')^* \not= \emptyset}} \widetilde{\lambda}^{-1}(\alpha^\perp_\Z; s) \widetilde{\mu}^{-1}(\alpha^\perp_\Z, \nM'', \kappa''; s)  \]
just reflects (in view of Theorem \ref{HEIGHTSPECIAL}) the following elementary equality of cycles:
\[ \nZ(\nL_\Z, \nM_\Z, \kappa) = \sum_{\substack{\SO'(\nL_\Zh)\alpha \subset \nIsome(\nM', \nL)(\Af) \cap \kappa' \\ \kappa'' \cap \alpha_\Af^\perp  \otimes (\nM_\Af'')^* \not= \emptyset}} \nZ(\alpha^\perp_\Z, \nM_\Z'', \kappa'').  \]
\end{PAR}

\begin{PAR}\label{SCHEITERN}
In the case of the modular curve, i.e. if $\nL$ is the lattice discussed in section \ref{MODULARCURVE}, we have explicitly by Kronecker's limit formula:
\[ 2 E(S; s+1) = \vol(\nSh({}^{K'}\nO(S^\perp)) + \hght(\nSh({}^{K'}\nO(S^\perp))) s + O(s^2) \]
(where the $K'$ denote the respective discriminant kernels $\SO'(S^\perp_\Zh)$) and in the sum over all $\SO'(\nL_\Zh)$ orbits:
\[ 2E(<q>, \kappa; s+1) = \vol(\nZ(\nL_\Z, <q>, \kappa)) + \hght(\nZ(\nL_\Z, <q>, \kappa)) s + O(s^2) \]
Therefore, Theorem \ref{HEIGHTSPECIAL} in this case (its geometric part is then roughly the classical class number formula) follows from 
Theorem \ref{TRACESOFEISENSTEIN}, 5. Observe, that the latter identity is an identity of functions in $s$, 
not only an identity of the first 2 Taylor coefficients! Therefore one might ask, if there is any arithmetic or $K$-theoretic ``explanation'' of
the equality of the higher terms in the Taylor expansion, too.
Observe furthermore that the derivative of $E(S; s+1)$ at $s=0$ is
{\em not} related in any simple way to the expansion of $4\lambda^{-1}(S^\perp_\Z; s)$ at $s=0$, if $S^\perp$ has 
nonsquarefree discriminant (cf. also \ref{SCHEITERN2}). Therefore one cannot expect a simple direct relation of the derivative of the orbit equation in this naive form 
to Arakelov geometry which is true without any restriction. 
\end{PAR}

\begin{PAR}
According to Theorem \ref{HEIGHTSPECIAL}, the value, resp. first derivative of $4\widetilde{\lambda}^{-1}(\nL_\Z)$ encode the geometric, resp. arithmetic volume of 
the orthogonal Shimura varieties associated with $\nL_\Z$ and its discriminant kernel, in the second case up to rational multiples of $\log(p)$ with either $p=2$ or
a prime with $p^2 | D(\nL_\Z)$. We will illustrate the computation of geometric and arithmetic volume of several Shimura varieties associated with lattices with square-free discriminant (outside 2) in the rest of this section.
We use the explicit computation (\ref{EXPLIZIT}, \ref{EXPLIZIT2}) and bring them to a
form involving derivatives of $L$-series at negative integers as it is common in the literature. 
\end{PAR}

\begin{BEISPIEL}[Heegner points]\label{BEISPIEL2DIM}
Let $\nL_\Z$ be a {\em two dimensional} negative definite lattice with square-free discriminant $D>0$ (in particular $2\nmid D$ and $-D$ is automatically fundamental). We have
\begin{eqnarray*}
\widetilde{\lambda}^{-1}(\nL_\Z; s) &=&  L(\chi_{-D}, 0) \\
 && +L(\chi_{-D}, 0) \cdot \left(-\frac{L'(\chi_{-D}, 0)}{L(\chi_{-D}, 0)} + \frac{1}{2}C - \frac{1}{2}\log(D) + \frac{1}{2}\log(2) \right) s \\
 && +O(s^2)
\end{eqnarray*}
\end{BEISPIEL}

\begin{BEISPIEL}[Modular curves]\label{BEISPIELMODULARCURVE}
Let $\nL_\Z$ be a {\em three dimensional} lattice with form $x_1 x_2 - \varepsilon x_3^2$, $\varepsilon$ square free, $2\nmid \varepsilon$. The
discrimiant $D$ is $2\varepsilon$.
\begin{eqnarray*}
 \widetilde{\lambda}^{-1}(\nL_\Z; s) &=& -\frac{1}{2}\zeta(-1)\prod_{p|\varepsilon} (p+1) - \frac{1}{2}\zeta(-1) \cdot\\
 && \cdot\prod_{p|\varepsilon} (p+1) \left(-2\frac{\zeta'(-1)}{\zeta(-1)}+\frac{1}{2}\sum_{p|\varepsilon}\frac{p-1}{p+1}\log(p)-1+C + \frac{1}{2}\log(2)\right) s \\
 && +O(s^2)
\end{eqnarray*}
\end{BEISPIEL}

\begin{BEISPIEL}[Shimura curves]\label{SHIMURACURVE}
Let $\nL_\Z$ be a {\em three dimensional} lattice again with discriminant $D=2\varepsilon$, $2\nmid \varepsilon$, $\varepsilon$ square-free, and assume that the form is {\em anisotropic at all $p | \varepsilon$}.
\begin{eqnarray*}
 \widetilde{\lambda}^{-1}(\nL_\Z; s) &=& -\frac{1}{2}\zeta(-1)\prod_{p|\varepsilon} (p-1) - \frac{1}{2}\zeta(-1)\cdot \\
 && \cdot \prod_{p|\varepsilon} (p-1)\left(-2\frac{\zeta'(-1)}{\zeta(-1)}+ 
\frac{1}{2}\sum_{p|\varepsilon}\frac{p+1}{p-1}\log(p)-1+C + \frac{1}{2}\log(2) \right) s \\
 && +O(s^2)
\end{eqnarray*}
\end{BEISPIEL}

\begin{BEISPIEL}[Hilbert modular surfaces]
Let $\nL_\Z$ a {\em four dimensional} lattice, being an orthogonal direct sum of a two dimensional indefinite of discriminant $D<0$,
square free and a hyperbolic plane. The we have:
\begin{eqnarray*}
 \widetilde{\lambda}^{-1}(\nL_\Z; s) &=& \frac{1}{4}\zeta(-1)L(\chi_{-D}, -1)  \\
 && +\frac{1}{4}\zeta(-1)L(\chi_{-D}, -1)\cdot \\
 && \cdot \frac{1}{4}\left(-2\frac{\zeta'(-1)}{\zeta(-1)}-\frac{L'(\chi_{-D}, -1)}{L(\chi_{-D}, -1)}-\frac{3}{2}-\frac{1}{2}\log(D)+ \frac{3}{2}C + \frac{1}{2} \log(2) \right) s \\
 && + O(s^2)
\end{eqnarray*}
\end{BEISPIEL}

\begin{BEISPIEL}[Siegel threefolds]
Let $\nL_\Z$ be a {\em five dimensional} lattice, which is the orthogonal sum of the negative of a three dimensional as in example~
(\ref{SHIMURACURVE}) and a hyperbolic plane. Let $D=2\varepsilon > 0$ be the discriminant of $\nL_\Z$
\begin{eqnarray*}
 \widetilde{\lambda}^{-1}(\nL_\Z; s) &=& -\frac{1}{4}\zeta(-1)\zeta(-3)\prod_{p|D}(p^2-1) -\frac{1}{4}\zeta(-1)\zeta(-3)\prod_{p|D}(p^2-1)\cdot \\
 && \cdot \left(-2\frac{\zeta'(-1)}{\zeta(-1)}-2\frac{\zeta'(-3)}{\zeta(-3)}+\frac{1}{2}\sum_{p | D}\frac{p^2+1}{p^2-1}\log(p) - \frac{17}{6} + 2C+\frac{1}{2}\log(2) \right) s \\
 && +O(s^2)
\end{eqnarray*}
\end{BEISPIEL}

\begin{BEISPIEL}[A 10-dimensional Shimura variety] Especially simple is the situation for a
Shimura variety of orthogonal type associated with an unimodular lattice (this has good reduction everywhere, except possibly $p=2$).
Let for example $\nL_\Z$ be the orthogonal direct sum of a positive definite $E_8$-lattice
and 2 hyperbolic planes. Here we get:
\begin{eqnarray*}
 \lambda^{-1}(\nL_\Z; s) &=& \frac{1}{16}\zeta(-1)\zeta(-3)\zeta^2(-5)\zeta(-7)\zeta(-9) \\
 && + \frac{1}{16}\zeta(-1)\zeta(-3)\zeta^2(-5)\zeta(-7)\zeta(-9) \cdot \left( -2\frac{\zeta'(-1)}{\zeta(-1)}-2\frac{\zeta'(-3)}{\zeta(-3)}\right. \\
 && \left.-3\frac{\zeta'(-5)}{\zeta(-5)}-2\frac{\zeta'(-7)}{\zeta(-7)}-2\frac{\zeta'(-9)}{\zeta(-9)}-\frac{14717}{1260}+\frac{11}{2}C \right) s \\
 &+& O(s^2)
\end{eqnarray*}
(and here this coincides with $\widetilde{\lambda}^{-1}$).
\end{BEISPIEL}

\section{The example $ac-b^2$}\label{MODULARCURVE}

\begin{PAR}
In this section, we investigate the special case of the following $\Z$-lattice:

\[ \nL_\Z =  \{ A \in M_2(\Z) \where {}^t A = A \} \]
with quadratic form $Q: \nL_\Z \rightarrow \Z$, $A \mapsto \det(A)$. It has discriminant 2 and signature (1,2).
The group scheme $\GL_2$ acts on $\nL$ by conjugation. This identifies $\PGL_2$ with the special orthogonal group scheme $\SO(\nL)$ of the lattice.
The discriminant kernel $\SO'(\nL_\Z)$ is equal to the special orthogonal group $\SO(\nL_\Z)$.
\end{PAR}

\begin{PAR}\label{normalizationS}
Via $\perp$, maximal negative definite sublattices $N \cap \nL_\Z$ correspond bijectively to vectors 
$S = \left( \begin{matrix} a&b\\ b&c  \end{matrix}\right)$ of positive length
$ac-b^2>0$ with $a<0, c<0$, $2|a, 2|c$ and $(\frac{a}{2},b,\frac{c}{2})=1$. They correspond also
to 2 complex vectors of length 0
\[ Z_{\tau^\pm} = \left( \begin{matrix} 1 &\tau^\pm\\\tau^\pm&(\tau^\pm)^2 \end{matrix}\right), \]
too, with $\tau^+ \in \HH, \tau^- = \overline{\tau^+}$, by virtue of

\begin{gather*}
 a (\tau^\pm)^2 - 2b \tau^\pm + c = 0 , \text{i.e. } \tau^\pm = \frac{b \pm i\sqrt{\det(S)}}{a} \\
 N_\tau = Z_{\tau^+} \oplus Z_{\tau^-} = \R \M{\frac{a}{2}&0\\0&-\frac{c}{2}} + \R \M{b&\frac{c}{2}\\\frac{c}{2}&0}
\end{gather*}
\end{PAR}

\begin{BEM}
The lattice $N_\tau \cap \nL_\Z = S^\perp_\Z$ is not necessarily equivalent to the lattice $\Z^2$ with form $x \mapsto \frac{1}{2} {}^t x S x$. 
However, the discriminant of $N_\tau \cap \nL_\Z$ is equal to $\det(S)$ if $S$ satisfies the 
conditions of \ref{normalizationS} (the generators of $N$ given above span a lattice of discriminant $\frac{c^2}{4}\det(S)$ and of index $\frac{c}{2}$ in a primitive one).
\end{BEM}

\begin{PAR}
Let
\[ E'(h,s) = \sum_{g \in P(\Q) \backslash \GL_2(\Q) } \Psi(gh)(s) \]
be the standard Eisenstein series of weight 0. Here $\Psi=\prod_\nu \Psi_\nu$, the $\Psi_\nu$s are the standard sections:
\[ \Psi_\nu(\M{\alpha_1& \\ & \alpha_2}u(\beta)k) = |\alpha_1|_\nu^{s} |\alpha_2|_\nu^{-s}, \]
however with different normalization than in (\ref{EISENSTEIN}).
We have 
\[ E'(g_\tau,s) = \frac{1}{2} \sum_{g \in \Gamma_\infty \backslash \SL_2(\Z) } |\Im(g \circ \tau)|^s =
\frac{1}{2} \sum_{(m,n)\in \Z \atop \ggT(n,m)=1} \frac{|y|^s}{|m\tau+n|^{2s}} \]
for $g_{x + yi} = \M{y & x \\ 0 & 1}$ and $\Gamma_\infty = \{ \M{1&x\\ & 1} \where x \in \Z \}$.

The Eisenstein series (in this normalization) satisfies the functional equation \cite[Theorem 1.6.1]{Bump}:
\[  Z(2s) E'(g, s) = Z(2-2s) E'(g, 1-s). \]
where $Z(s) = \zeta(s)\Gamma(\frac{1}{2}s)\pi^{-\frac{1}{2}s}$ is the normalized Riemann zeta function, satisfying $Z(s)=Z(1-s)$.
We normalize the Eisenstein series as follows:
\[ E(z; s) := \frac{Z(2s)}{Z(s)}E'(z; s). \]

For each integer $q$ and $\kappa\in \nL_\Zh^*/\nL_\Zh$, we consider
the decomposition of $\nIsome(<q>, \nL)(\Af) \cap \kappa$ into orbits $\SO'(\nL_\Zh)S$ (w.l.o.g. $S \in \nL_\Q$).  
We are interested in the `partial traces': 
\[ E(S, s) := \sum_{g_f \in \SO(S^\perp) \backslash \SO(S^\perp_\Af) / \SO'(S^\perp_\Zh)} \frac{1}{\#\SO'(S^\perp_\Z)} E(g_{\tau} g_f, s) \]
(here $\SO(S^\perp_\Q)$ is a torus in $\SO(\nL_\Q)$ and $\tau \in \HH$ is such that $Z_\tau \perp S$, observe $E(g_{\tau} g_f, s)=E(g_{\overline{\tau}} g_f, s)$) and in the
`complete traces':
\[ E(<q>, \kappa; s) := \sum_{\SO'(\nL_\Zh)S \subset \nIsome(<q>, \nL)(\Af) \cap \kappa} E(S; s). \]
Since $\nL_\Z$ has class number 1, w.l.o.g. $S \in \nIsome(<q>, \nL)(\Q)$ and the latter can also be written as follows:
\[ E(<q>, \kappa; s) = \sum_{\SO(\nL_\Z)S \subset \nIsome(<q>, \nL)(\Q) \cap \kappa } \frac{1}{\#\SO'(S^\perp_\Z)} E(g_{\tau}; s).  \]

The aim of this section is to prove the following (well-known)
\end{PAR}

\begin{SATZ}\label{TRACESOFEISENSTEIN}
\begin{eqnarray*}
E(S, s) &=& \prod_\nu E_\nu(S, s), \\
E(<q>, \kappa; s) &=& \prod_\nu E_\nu(<q>, \kappa; s), 
\end{eqnarray*}
where $E_\infty(<q>, \kappa; s) = E_\infty(S, s)$ is independent of $S \in \nIsome(<q>, \nL)(\Q)\cap \kappa$ and
$E_p(<q>, \kappa; s)$ is defined by the orbit equation:
\begin{equation}\label{trueorbitequation} 
E_p(<q>, \kappa; s) := \sum_{\SO'(\nL_\Zp)S_p \subset \nIsome(<q>, \nL)(\Qp) \cap \kappa_p } E_p(S_p, s). 
\end{equation}
We have explicitly $E_p(S, s) = |D(S^\perp_\Zp)|_p^{-\frac{s}{2}}\frac{\zeta_p(S^\perp_\Zp; s)}{\zeta_p(s)}$, where 
$\zeta_p(S^\perp_\Zp, s)$ is the normalized zeta function (cf. Definition \ref{localzeta}) of the lattice $S^\perp_\Zp$ and $\zeta_p(s) = \frac{1}{1-p^{-s}}$.
\end{SATZ}

The relation to representation densities is determined by 

\begin{SATZ}\label{TRACESOFEISENSTEIN2}\quad\\
\begin{enumerate}
\item 
$E_\infty(<q>; s) = E_\infty(S; s) = 2 \lambda^{-1}_\infty(S^\perp_\Z; s-1)$
\item If $p\not=2$, 
$ E_p(<q>; s) = \widetilde{\lambda}^{-1}_p(\nL_\Zp; s-1) \widetilde{\mu}_p(\nL_\Zp, <q>; s-1)  $
\item If $p\not=2$, $\nu_p(q)\le 1$, we have:
\[ E_p(S; s) = E_p(<q>; s) = \widetilde{\lambda}^{-1}_p(S^\perp_\Z; s-1) \]
\item If $p\not=2$, the $E_p(S; s)$ satisfy the following functional equation:
\[ \frac{E_p(S; 1-s)}{L_p(\chi_{-q};1-s)} = \frac{E_p(S; s)}{L_p(\chi_{-q};s)}.  \]
\item In the product:
\[ E(<q>, \kappa; s) = 2 \widetilde{\lambda}^{-1}(\nL_\Zp; s-1) \widetilde{\mu}(\nL_\Zp, <q>, \kappa; s-1) \]
and if $q$ is square-free:
\[ E(S; s) = E(<q>; s) = 2 \widetilde{\lambda}^{-1}(S^\perp_\Z; s-1) \]
in both cases maybe up to a rational function in $2^{-s}$.
\end{enumerate} 
\end{SATZ}

Recall that $\widetilde{\mu}(\nL_\Zp, <q>, \kappa; s)$ itself is the `holomorphic part' of a Fourier coefficient (of index $q$) of another, metaplectic Eisenstein series \ref{GLOBALEISENSTEIN}.

\begin{proof}[Proof of Theorem \ref{TRACESOFEISENSTEIN}]
If $E(<q>, s)$ is defined via equation (\ref{trueorbitequation}), the second product expansion follows immediately from the first and the definition. 
Hence it is enough to show
\[ E(S, s) = \frac{\Gamma(\frac{s}{2}+\frac{1}{2}) } {\pi^{\frac{s}{2}+\frac{1}{2}}} |\det(S)|_\infty^{\frac{s}{2}} \frac{\zeta(S^\perp_\Z; s)}{\zeta(s)}. \]
for $S$ normalized as in \ref{normalizationS}.

First we have
\begin{gather*}
E'(S; s) =  \sum_{h_f \in \SO(S^\perp_\Q) \backslash \SO(S^\perp_\Af) / \SO'(S^\perp_\Zh)} \frac{1}{\#\SO'(S^\perp_\Z)} 
 \sum_{ g \in P(\Q) \backslash \GL_2(\Q) } \Psi(s; g g_\tau h_f) \\
= \sum_{h \in \SO(S^\perp_\Af) / \SO'(S^\perp_\Zh)} \Psi(s; h g_\tau),
\end{gather*}
where we used $\PGL_2(\Q) = P(\Q) \SO(S^\perp_\Q)$ in the last step. Considering the right invariance of $\Psi(s)$ under $K$, this may be rewritten as:
\begin{gather*}
E'(S; s) = \Psi_\infty(s; g_\tau) \prod_p \frac{  \int_{ \Gm(\Qp) \backslash T(\Qp)} \Psi_p(s; x) \mu(x)} {\vol_\mu(K \cap T(\Af))},
\end{gather*}
where we denoted the preimage of $\SO(S^\perp)$ in $\GL$ by $T$
and $\mu$ is any translation invariant measure on $\Gm(\Qp) \backslash T(\Qp)$.

Lemma \ref{LEMMAIWASAWA} shows that it is convenient to parameterize the torus $T$ as follows (recall $a\not=0$):
\begin{eqnarray*}
 \iota: \Q_p^2 -\{ 0 \} &\rightarrow& T(\Qp)  \\
 \left(\begin{matrix} x \\ y \end{matrix} \right) &\mapsto& \left(\begin{matrix} x + 2\frac{b}{a}y & \frac{c}{a} y \\ -y & x \end{matrix} \right)
\end{eqnarray*}
We have $\det(\iota(v)) = \frac{Q(v)}{\frac{a}{2}}$, where $Q(v) = \frac{a}{2}x^2 + bxz + \frac{c}{2}y^2$.
A translation invariant measure on $T(\Af)$ is given by $\det(\iota(v))^{-1} \dd v$, where
$\dd v$ is the standard measure on $\Qp^2$. We choose the measure $\mu$, quotient of this by a measure giving 
$\Zp^*$ volume 1.

This yields (using the lemma):
\[ \int_{\Z_p^2} |\frac{Q(v)}{\frac{a}{2}}|_p^{s-1} \dd v = \sum_{i=0..\infty} p^{-2is} \int_{ \Gm(\Qp) \backslash T(\Qp)} \Psi_p(s; x) \mu \]
and therefore
\begin{gather*}
E'(S; s) = \frac{2^{-s}\left|\det(S)\right|_\infty^{\frac{1}{2}s}}{ \zeta(2s) } \\ \prod_p \frac{ \int_{\Z_p^2} |Q(v)|_p^{s-1} \dd v }
 { |\frac{a}{2}|_p^{-1} \vol\{x, y \where \frac{c}{a}y, x, x+2 \frac{b}{a} y, y \in \Zp, x^2+2\frac{b}{a}xy+\frac{c}{a}y^2 \in \Zp^*\}. }
\end{gather*}
Substituting $\frac{a}{2}y$ for $y$ and then $x-\frac{b}{2}y$ for $x$ in the volume computation in the denominator, we get:
\begin{gather*}
 |\frac{a}{2}|_p^{-1} \vol\{\cdots \} = \vol\{x, y \in \Zp^2 \where x^2+\frac{ac-b^2}{4}y^2 \in \Zp^* \},
\end{gather*}
if $b$ is even or $p \not=2$, otherwise substitute in addition $x + \frac{1}{2}$ for $x$ and get
\begin{gather*}
|\frac{a}{2}|_p^{-1} \vol\{\cdots \} = \vol\{x, y \in \Zp^2 \where x^2+xy+\frac{ac-b^2+1}{4}y^2 \in \Zp^* \}.
\end{gather*}
Hence in any case:
\begin{gather*}
E'(S; s) =  2^{-s} \frac{\left| \det(S) \right|_\infty^{\frac{1}{2}s}}{ \zeta(2s) } \prod_p \zeta_p(S^\perp_\Zp, s),
\end{gather*}
where $\zeta_p(S^\perp_\Zp, s)$ is the normalized local zeta function of the lattice $S^\perp_\Zp$ (see~\ref{localzeta}). 
Note, that it depends only on
the discriminant for 2 dimensional lattices.

Hence 
\[ E(S; s) = \frac{Z(2s)}{Z(s)}E'(S; s) = 2^{-s} \frac{\Gamma(s)\pi^{\frac{s}{2}}}{\pi^{s}\Gamma(\frac{s}{2})} |\det(S)|^{\frac{s}{2}} \frac{\zeta(S^\perp_\Z; s)}{\zeta(s)}. \]

Applying the doubling formula for the $\Gamma$-function, we get indeed
\[ E(S; s) = \frac{\Gamma(\frac{s}{2}+\frac{1}{2}) } {\pi^{\frac{s}{2}+\frac{1}{2}}} |\det(S)|^{\frac{s}{2}} \frac{\zeta(S^\perp_\Z; s)}{\zeta(s)}. \]

\end{proof}

\begin{LEMMA}\label{LEMMAIWASAWA}
We have $\Psi_\infty(s; g_\tau) = \left|\frac{\sqrt{\det(S)}}{a}\right|_\infty^{s}$ and
$\Psi_p(s; h) = \left|\frac{\det(h)}{\ggT(h_{12},h_{22})^2}\right|_p^{s}$.
\end{LEMMA}
\begin{proof}
The first assertion follows from the evaluation of the relation of $g_\tau$ and $S$, given above. 
For the second assertion consider the case $\nu_p(h_{12}) > \nu_p(h_{22})$, hence $ggT(h_{12},h_{22})=h_{22}$.
\[ \left( \begin{matrix} 1 & -h_{21}det(h)^{-1}h_{22}\\ & 1 \end{matrix}\right)
\left( \begin{matrix} \det(h)^{-1}h_{22} &  \\ & h_{22}^{-1} \end{matrix}\right)
\left( \begin{matrix} h_{11} & h_{21} \\ h_{12} & h_{22} \end{matrix} \right) =
\left( \begin{matrix} 1 &  \\ \frac{h_{12}}{h_{22}} & 1 \end{matrix} \right)
\]
where the matrix on the right hand side is integral. Analogiously for the case $\nu_p(h_{12}) <\nu_p(h_{22})$.
\end{proof}

\begin{proof}[Proof of Theorem \ref{TRACESOFEISENSTEIN2}]
1. came out in the proof of Theorem \ref{TRACESOFEISENSTEIN}.

2. Let $p \not= 2$, $S \in \nL_\Zp$ and $Q(S)=q$. Denote $l:=\nu_p(q)$.
We have to show 
\[ \sum_{\SO'(\nL_\Zp)S \subset \nIsome(<q>, \nL)(\Qp) \cap \kappa} |d(S^\perp_\Zp)|_p^{-\frac{s}{2}} \frac{\zeta_p(S^\perp_\Zp; s)}{\zeta_p(s)} = \widetilde{\lambda}^{-1}_p(\nL_\Zp; s-1) \widetilde{\mu}_p(\nL_\Zp, <q>, \kappa; s-1) \]
Using the orbit equation, this can be rewritten as:
\begin{gather*} 
\sum_{\SO'(\nL_\Zp)S \subset \nIsome(<q>, \nL)(\Qp) \cap \kappa} |d(S^\perp_\Zp)|_p^{-\frac{s}{2}} \frac{\zeta_p(S^\perp_\Zp; s)}{\zeta_p(s)} \\
= |q|_p^{\frac{1}{2}(s-1)} \sum_{\SO'(\nL_\Zp)S \subset \nIsome(<q>, \nL)(\Qp) \cap \kappa} \lambda^{-1}(S^\perp_\Zp; s-1)
\end{gather*}
In case $l=0$, we get:
\begin{eqnarray*} 
\zeta_p(S^\perp_\Zp; s) &=& \frac{1}{(1-X)(1-\left(\frac{-\varepsilon'}{p}\right)X)} \\
\lambda_p^{-1}(S^\perp; s-1) &=& \frac{1}{1-\left(\frac{-\varepsilon'}{p}\right)X}
\end{eqnarray*} 
and we have just 1 orbit, hence the equation is true.

In case $l$ odd, we have $\frac{l+1}{2}$ orbits (Lemma \ref{LEMMA1ORBIT}) and get (Theorem \ref{EXPLIZIT} and \ref{ZETAEXPLIZIT}):
\begin{eqnarray*} 
\sum_S |d(S^\perp_\Zp)|_p^{-\frac{s}{2}} \zeta_p(S^\perp_\Zp; s) &=& \sum_{k=0}^{\frac{l-1}{2}} X^{-\frac{1}{2}-k} \frac{1-(pX^2)^{k+1}-X+X(pX^2)^k}{(1-X)(1-pX^2)}, \\
\sum_S \lambda_p^{-1}(S^\perp_\Zp; s-1) &=& \sum_{k=0}^{\frac{l-1}{2}} (pX)^{-2k-1}p^{k+\frac{1}{2}} 
\end{eqnarray*} 
Hence the equation reduces to
\begin{gather*} 
 \sum_{k=0}^{\frac{l-1}{2}} X^{-\frac{1}{2}-k} \frac{1-(pX^2)^{k+1}-X+X(pX^2)^k}{1-pX^2} = (pX)^{\frac{l}{2}} \sum_{k=0}^{\frac{l-1}{2}} (pX)^{-2k-1}p^{k+\frac{1}{2}} 
\end{gather*}
which may be checked in an elementary way.

In case $l \ge 2$ even, we have $\frac{l}{2}+1$ orbits (Lemma \ref{LEMMA1ORBIT}) and get (Theorem \ref{EXPLIZIT} and \ref{ZETAEXPLIZIT}):
\begin{eqnarray*} 
\sum_S |d(S^\perp_\Zp)|_p^{-\frac{s}{2}} \zeta_p(S^\perp_\Zp; s) &=& \frac{1}{(1-X)(1-\left(\frac{-\varepsilon'}{p}\right)X)} \\
&+& \sum_{k=1}^{\frac{l}{2}} X^{-k} \left( \frac{p^{k}X^{2k} -p^{k-1}X^{2k-1} }{(1-X)(1-(\frac{-\varepsilon}{p})X)} \right.  \\
&& +\left. \frac{1-(pX^2)^{k}-X+X(pX^2)^{k-1} }{(1-X)(1-pX^2)} \right) \\
\sum_S \lambda^{-1}(S^\perp; s-1) &=& \frac{1}{1-\left(\frac{-\varepsilon'}{p}\right)X} + \sum_{k=1}^{\frac{l}{2}} (pX)^{-2k}p^{k} 
\end{eqnarray*} 
Hence the equation reduces to
\begin{gather*} 
  \frac{1}{1-\left(\frac{-\varepsilon'}{p}\right)X} + \sum_{k=1}^{\frac{l}{2}} X^{-k} \left( \frac{p^{k}X^{2k} -p^{k-1}X^{2k-1} }{1-(\frac{-\varepsilon}{p})X} +
\frac{1-(pX^2)^{k}-X+X(pX^2)^{k-1} }{1-pX^2} \right) \\
 = (pX)^{\frac{l}{2}} \left( \frac{1}{1-\left(\frac{-\varepsilon'}{p}\right)X} + \sum_{k=1}^{\frac{l}{2}} (pX)^{-2k}p^{k} \right)  
\end{gather*}
which may again be checked in an elementary way.

3. is seen by a comparison of the formul\ae\ obtained for $l=0$ and $l=1$ in the proof of 2. above with the explicit formul\ae\ for $\lambda$ in Theorem \ref{EXPLIZIT}. 

4. is true in the product by means of the functional equation of the Eisenstein series. The local statement follows because the functions $p^{-s}$ are algebraically independent for different primes. Alternatively one can check this equation from the explicit formul\ae\ obtained in the proof of 2.

5. is obtained out of 2. and 3. by taking the product over all $\nu$.
\end{proof}

\begin{BEM}\label{SCHEITERN2}
Theorem \ref{TRACESOFEISENSTEIN2}, 2. involves an identity between sums over orbits in $\nIsome(<q>, \nL_\Zp)$. These identities are striking from an elementary point of view, since the individual terms do not match. The terms coming from the traces over the Eisenstein series are directly related to 
the heights of individual sub-Shimura varieties of the modular curve (cf. also \ref{SCHEITERN}). 
Hence, in this case, the orbit equation we derived in this work cannot be related to 
Arakelov geometry (including information from $p$), if there is more than one orbit (at $p$). 
\end{BEM}

\appendix

\def\thesection{\Alph{section}} 

\section{Some basic facts about quadratic forms}

\begin{LEMMA}\label{DISKRIMINANTENKERN}
Let $X \subset \nL_\Zp^*$ be a subset and $\Stab(X) \subset \SO'(\nL_\Zp)$ the {\em pointwise} stabilizer.
Assume that $X_\Qp$ is non-degenerate. Then we have
\[ \Stab(X) = \SO'(X_\Zp^\perp), \]
where $X^\perp_\Zp = \{ v \in \nL_\Zp \where \langle v, w \rangle = 0 \ \forall w\in X \}$.
\end{LEMMA}
\begin{proof}
Let $\beta \in \SO'(\nL_\Zp)$, i. e. $\beta v-v \in \nL_\Zp$ for all $v \in \nL_\Zp^*$.
If we have $\beta w=w$ for all $w \in X$, one can consider $\beta$ via restriction as an element of $\SO(X_\Zp^\perp)$.
If moreover $w^\perp \in (X_\Zp^\perp)^*$, we find a $v \in \nL_\Zp^*$ satisfying $v = w + w^\perp$ with some
$w \in <X>_\Qp$ ($X_\Zp^\perp$ is a primitive sublattice).
This yields $\beta w^\perp - w^\perp = \beta v - v \in \nL_\Zp$. Hence $\Stab(X) \subseteq \SO'(X_\Zp^\perp)$.

On the other hand, let $\beta \in \SO'(X_\Zp^\perp)$. It can be extended uniquely to an element
$\beta \in \SO(\nL_\Qp)$, fixing $<X>_\Qp$ pointwise.
We claim that this extension lies in fact in $\SO(\nL_\Zp)$. For, let $v \in \nL_\Zp$ be given and
write $v = w^\perp + w$. Then we have $w^\perp \in (X_\Zp^\perp)^*$. hence
$\beta v - v = \beta w^\perp - w^\perp \in \nL_\Zp$.

Now suppose $v \in \nL^*_\Zp$. Then we have still $w^\perp \in (X^\perp_\Zp)^*$ because
$\langle v, {w^\perp}'\rangle = \langle w^\perp, {w^\perp}'\rangle \in \Zp$
for all ${w^\perp}' \in X^\perp_\Zp$. Hence $\beta v - v \in \nL_\Zp$ as well, which means $\beta \in \SO'(\nL_\Zp)$
\end{proof}

\begin{LEMMA}\label{KITAOKATSSS}
Let $\nM_\Zp$ an {\em unimodular} sublattice of $\nL_\Zp$. Then we have
\[ \nL_\Zp = \nM_\Zp \perp \nM_\Zp^\perp. \]
\end{LEMMA}
\begin{proof}
Follows from \cite[Prop. 5.2.2]{Kitaoka}.
\end{proof}

\begin{LEMMA}{\cite[Theorem 5.2.2]{Kitaoka}}\label{KITAOKAT522}
If $\nL_\Zp$ is unimodular, we have
\[ \nL_\Zp \simeq H^r_\Zp \perp \nL_\Zp^0, \]
with $\nL_\Zp^0$ anisotropic. Here $H_\Zp$ is an hyperbolic plane.
\end{LEMMA}

The following is well-known:
\begin{LEMMA}\label{KITAOKARRR}
Let $R$ be a discrete valuation ring with $|2|=1$ or a field and $\nL_R$ a lattice with quadratic form $Q_\nL \in \Sym^2(\nL_K^*)$, where $K$ is the quotient field of $R$.
There is a basis $e_1, \dots, e_m$ of $\nL_R$, with respect to which the $Q_\nL$ is given by
\[ Q_\nL: x \mapsto \sum_i \varepsilon_i p^{\nu_i} x_i^2, \]
where $\varepsilon_i \in \Zpp^*, \nu_i \in \Z$ and $\nu_1 \le \cdots \le \nu_m$.
\end{LEMMA}

\begin{LEMMA}\label{ZYKLISCH}
Let $p\not=2$. Assume $\nL_\Zp^*/\nL_\Zp$ is cyclic. Then 
\[ \SO(\nL_\Zp)/\SO'(\nL_\Zp) = \begin{cases} 1 & \text{if } \nu=0 \text{ or } m=1, \\ \Z/2\Z & \text{otherwise, } \end{cases} \]
where $p^\nu = |D(\nL_\Zp)|^{-1}$ is the order of $\nL_\Zp^*/\nL_\Zp$.
\end{LEMMA}
\begin{proof}
In the representation given by Lemma~\ref{KITAOKARRR}
$\nu_m$ is equal to $\nu$ and all other $\nu_i$ vanish.
Hence $v := p^{-\nu}e_m$ is a generator of $\nL_\Zp^*/\nL_\Zp$ and $Q(v) = \varepsilon_m p^{-\nu}$.
Let $v'$ be its image under an arbitrary isometry.
$v'$ has a representation
\[ v' = \sum_{i<m} \alpha_i e_i + p^{-\nu} \alpha_m e_m \qquad \alpha_i \in \Zp \]
and
\[ Q(v') = \sum_{i<m} \varepsilon_i \alpha_i^2 + \varepsilon_m p^{-\nu} \alpha_m^2 = \varepsilon_m p^{-\nu}. \]
From this it follows
\[ \alpha_m^2 \equiv 1 \mod p^{\nu}, \]
hence $(p \not= 2)$
\[ \alpha \equiv \pm 1 \mod p^{\nu}. \]
The occurring sign defines a character of the orthogonal group. An element is
in its kernel, precisely if it is in the discriminant kernel. Moreover, if $m>1$ there are elements in 
$\SO$, which yield sign $-1$, for example composition of reflection along $e_m$ and any $e_i, i<m$.
\end{proof}

\begin{LEMMA}\label{KITAOKAC531}
Assume $p\not=2$ and let
\[ \nL_\Zp = \nM_\Zp \perp \nM'_\Zp = \nN_\Zp \perp \nN_\Zp', \]
with $\beta: \nM_\Zp \cong \nN_\Zp$. Then we have
\[ \nM_\Zp' \cong \nN_\Zp'. \]
In particular, there exists an isometry $\alpha \in \SO(\nL_\Zp)$ with $\alpha(\nM_\Zp)=\nN_\Zp$.
If $\nM_\Zp$ is unimodular, we may choose $\alpha \in \SO'(\nL_\Zp)$. If $\nM'_\Zp$ has a vector of unit length, we may 
assume in addition, that $\alpha|_{\nM_\Zp} = \beta$.
\end{LEMMA}
\begin{proof}
The first part of the assertion is shown in \cite[Corollary 5.3.1]{Kitaoka}. It remains to see that we may choose the
isometry in the discriminant kernel, if $\nM_\Zp$ is unimodular: For this we proceed by induction on the dimension on $\nM_\Zp$.
If $\nM_\Zp$ is one dimensional, let $v$ be a generating vector of unit length
and $v'$ its image under $\beta$. One of the vectors $v+v'$ or $v-v'$ has unit length, call it $w$. The reflection
along $w$ lies obviously in $O'(\nL_\Zp)$ and interchanges $<v>$ and $<v'>$. 
By composition with the reflection along $v'$, we may assume, that it lies in $\SO'(\nL_\Zp)$.  

Assume now $\dim(\nM_\Zp)>1$. Let $v$ be a vector of unit length in $\nM_\Zp$ and $v'$ its image. We have
\[ \nL_\Zp = <v> \perp v^{\perp } = <v'> \perp {v'}^{\perp}, \]
(Lemma \ref{KITAOKATSSS}) and
\[ v^{\perp \nL_\Zp} = v^{\perp \nM_\Zp} \perp \nM_\Zp' \qquad {v'}^{\perp \nL_\Zp} = {v'}^{\perp \nN_\Zp} \perp \nN_\Zp', \]
and we have an isometry (case above)
in the discriminant kernel, which maps $<v>$ to  $<v'>$, and hence $v^\perp$ to ${v'}^\perp$. $v^{\perp \nM_\Zp}$ and ${v'}^{\perp \nN_\Zp}$ are now isomorphic (again by the case above) hence (induction hypothesis), there is an isometry in $\SO'({v'}^\perp)$ mapping the image of $v^{\perp \nM_\Zp}$ to ${v'}^{\perp \nN_\Zp}$. It lifts to $\SO'(\nL_\Zp)$. Composition with the first isometry gives the induction step. 
The proof shows, that we may arrange $\alpha|_{\nM_\Zp} = \beta$ if there is a reflection in $\SO'(\nM'_\Zp)$.
\end{proof}

\begin{LEMMA}\label{LEMMA1ORBIT}
Let $\nL_\Zp$ $(\dim(\nL) \ge 3)$ be a unimodular lattice, $p \not=2$, and $q \in \Zp \setminus \{0\}$.

Then $\SO(\nL_\Zp)$ acts transitively on $\{ \alpha \in \nIsome(<q>, \nL_\Zp) \where \im(\alpha) \text{ is saturated} \}$.
In particular, it acts transitively on $\nIsome(<q>, \nL_\Zp)$ with $|q|_p=\frac{1}{p}$.
\end{LEMMA}
\begin{proof}
Take any $v$ with $Q_L(v)=q$.
Diagonalize the form (Lemma~\ref{KITAOKARRR}) and take the reflection $v'$ of $v$ at any basis vector $e_i$ with the property that $v_i \in \Z_p^*$ (this must exist,
since otherwise the vector would not be primitive).
We have $p \nmid \langle v, v' \rangle$. Therefore the form on $\Zp v \oplus \Zp v'$ is unimodular, hence $\Zp v \oplus \Zp v'$ is primitive and a direct summand by
Lemma~\ref{KITAOKATSSS}. It is necessarily a hyperbolic plane, since modulo $p$ it represents zero. We have shown that any primitive vector in $\nIsome(<q>, \nL)$ lies in a hyperbolic plane. Now use Lemma~\ref{KITAOKAC531} and the fact, that $\text{O}(H_\Zp)$ (not $\SO$!) acts transitively on primitive vectors of length $q$ on $H_\Zp$. 
\end{proof}

We see that for $p\not=2$, $j>0$, and an unimodular lattice $\nL_\Zpp$, there are precisely $\lfloor\frac{j}{2}\rfloor+1$ orbits of vectors of length $p^{j}$, 
indexed according to their `saturatedness'.

\bibliographystyle{model1b-num-names}
\bibliography{hoermann}
\biboptions{square,comma,sort}







\end{document}